\documentclass[sn-mathphys]{sn-jnl}

\jyear{2021}%

\theoremstyle{thmstyleone}%
\newtheorem{theorem}{Theorem}

\theoremstyle{thmstyletwo}%
\newtheorem{remark}{Remark}%
\newtheorem{proposition}{Proposition}%
\newtheorem{corollary}{Corollary}
\newtheorem{lemma}{Lemma}

\theoremstyle{thmstylethree}%
\newtheorem{definition}{Definition}%

\usepackage{latexsym, amsfonts, amscd, verbatim, amsmath,enumerate}  
\usepackage{mathtools}

\usepackage{comment}

\usepackage{amssymb}

\usepackage{varwidth}
\usepackage{tasks}

\usepackage{multicol}

\usepackage[capitalise,nameinlink]{cleveref}

\usepackage{enumitem}
\newlist{assumption}{enumerate}{1}					
\setlist[assumption]{label=(\textsc{a}\arabic*)}
\crefname{assumptioni}{Assumption}{Assumptions}

\newlist{assumption2}{enumerate}{1}
\setlist[assumption2]{label=(\textsc{a}6)}			
\crefname{assumptioni}{Assumption}{Assumptions}

\newlist{assumption3}{enumerate}{1}
\setlist[assumption3]{label=(\textsc{a}6')}			
\crefname{assumptioni}{Assumption}{Assumptions}

\newcommand{\Rbb}{\mathbb{R}}

\usepackage{bbm}
\newcommand{\1}{\mathbbm{1}}

\newcommand{\Linop}{\mathbb{L}}

\newcommand{\cone}{\text{cone}}
\newcommand{\norm}[1]{\|#1\|}
\newcommand{\proj}{\textrm{proj}}

\DeclareMathOperator{\Id}{\mathrm{Id}}

\newcommand{\dx}{\,\mathrm{d}x}
\newcommand{\ds}{\,\mathrm{d}s}
\newcommand{\barK}{K}
\newcommand{\Ha}{\mathcal{H}}           
\newcommand{\dH}{\,\mathrm{d}\mathcal{H}}
\newcommand{\barr}{{\bar{r}}}

\usepackage{xcolor}

\normalbaroutside  

\begin{document}
	\title[Optimality conditions and Lipschitz stability ...]{Optimality conditions and Lipschitz stability for non-smooth semilinear elliptic optimal control problems with sparse controls}
	
	
	\author*[1]{\fnm{Vu Huu} \sur{Nhu}}\email{nhu.vuhuu@phenikaa-uni.edu.vn}

	\author[1]{\fnm{Phan Quang} \sur{Sang}}\email{sang.phanquang@phenikaa-uni.edu.vn}
	\equalcont{These authors contributed equally to this work.}
	
	
	\affil*[1]{\orgdiv{Faculty of Fundamental Sciences}, \orgname{PHENIKAA University}, \orgaddress{\city{Yen Nghia, Ha Dong, Hanoi}, \postcode{12116}, \country{Vietnam}}}

	
	
\abstract{%
			This paper is concerned with first- and second-order optimality conditions as well as the stability for non-smooth semilinear optimal control problems 
			involving the $L^1$-norm
			of the control in the cost functional. 
			In addition to the appearance of 
			the $L^1$-norm 
			leading to the non-differentiability of the objective and promoting the sparsity of the optimal controls, the non-smoothness of the nonlinear coefficient in the state equation causes the same property of the control-to-state operator. Exploiting a regularization scheme, we derive $C$-stationarity conditions for any local optimal control. Under a structural assumption on the associated state, we define the curvature functional for the part not 
			including the $L^1$-norm 
			of controls of the objective for which the second-order necessary and sufficient optimality conditions are shown. Furthermore, 
			under a more restrictive structural assumption imposed on 
			the mentioned state, an explicit formulation of the curvature is established and thus the explicit second-order optimality conditions are stated. Finally, the Lipschitz stability of local solutions with respect to the sparsity parameter is shown. %
}



\maketitle

\small 

\section{Introduction}
In this paper, we consider the following non-smooth semilinear elliptic optimal control problem
\begin{equation}
	\label{eq:P}
	\tag{$P_\kappa$}
	\left\{
	\begin{aligned}
		&\min J_\kappa(u)\\
		& \alpha \leq u(x) \leq \beta \quad \text{{for a.a. }} x \in \Omega,
	\end{aligned}
	\right.
\end{equation}
where $J_\kappa(u) := F(u) + \kappa j(u)$ with $F: L^\infty(\Omega) \to \Rbb$ and $j: L^1(\Omega) \to \Rbb$ defined by
\begin{equation} \nonumber
	F(u) = \int_\Omega L(x,y_u(x)) \dx + \frac{\nu}{2} \int_\Omega |u(x)|^2 \dx \quad \text{and} \quad j(u) = \int_\Omega |u(x)| \dx.
\end{equation}
Here $y_u$ stands for the unique solution in $H^1_0(\Omega) \cap C(\overline\Omega)$ of the state equation
\begin{equation} \label{eq:state}
	\left\{
	\begin{aligned}
		A y + f(y) & = u && \text{in } \Omega, \\
		y &=0 && \text{on } \partial\Omega,
	\end{aligned}
	\right.
\end{equation}
where 
$A$ is the linear operator given by
\begin{equation} \nonumber
	Ay = -\sum_{i,j=1}^{{N}} \partial_{j}[a_{ij}(x) \partial_{i} y] + a_0(x)y
\end{equation} 
with 
$N \geq 1$ 
and $\partial_i$, $i = 1,\ldots, N$,  standing for the partial derivatives, 
and 
$f: \Rbb \to \Rbb$ is finitely $PC^2$ (continuous and $C^2$ apart from finitely many points) and defined as
\begin{equation} 
	\label{eq:PC1-rep}
	f(t) = \sum_{i =0}^{\barK} \1_{(\tau_{i},\tau_{i+1}]}(t)f_{i}(t)  \quad \text{for all } t \in \Rbb,
\end{equation}
with $-\infty =: \tau_0 < \tau_1 <\cdots <\tau_\barK < \tau_{\barK +1} := \infty$ 
and finitely many $C^2$ functions $f_i$, $0 \leq i \leq \barK$, that are monotonically increasing and satisfy the
following continuity condition
at $\tau_i:$ 
\begin{equation} \label{eq:PC1-continuous-cond}
	f_{i-1}(\tau_i) = f_{i}(\tau_i) \quad \text {for all } 1 \leq i \leq \barK.
\end{equation}
Here and in what follows, we use the convention $(t, +\infty] := (t, \infty)$ and the symbol $\1_M$ denotes the characteristic function of a set $M$, i.e., $\1_M(x) = 1$ if $x \in M$ and $\1_M(x) =0$ otherwise. 

\medskip
A typical example of the non-smooth state equation is the following 
\begin{equation}
	\label{eq:max-pde}
	-\Delta  y + \max(y,0)  = u \text{ in } \Omega, \quad y =0  \text{ on } \partial\Omega,
\end{equation}
which arises, 
for instance, 
in models of the deflection of a stretched thin membrane partially covered by water (see \cite{Kikuchi1984}).

\medskip 

The salient features of the optimal control problem \eqref{eq:P} are, of course, the non-differentiability of the nonlinearity coefficient $f$ in the state equation and the appearance of the $L^1$-norm of the controls. 
The first issue leads to the non-differentiability of the corresponding control-to-state operator; see \cref{prop:control-to-state} below, while the second one brings about the non-smoothness of the objective functional in the control variable and accounts for the sparsity of the optimal control. 
As a consequence, standard techniques based on the second-order differentiability of the control-to-state operators, e.g. as in \cite{CasasHerzogWachsmuth2012,CasasTroltzsch2020}, for deriving the first- and second-order optimality conditions of \eqref{eq:P}  are typically inapplicable, making the analytical and numerical treatment challenging.

\medskip

Let us put this research work into perspective. An optimal control problem with an $L^1$-term in the cost functional was originally investigated in \cite{Stadler2009}, where the state equation is linear. The second-order necessary and sufficient optimality conditions are not needed since this problem is convex. Later,   the semilinear elliptic optimal control problems were analyzed in \cite{CasasHerzogWachsmuth2012} for box control constraints, in \cite{CasasTroltzsch2014,CasasTroltzsch2020} for pointwise state constraints. The numerical analysis 
was
also studied in these papers except for \cite{CasasTroltzsch2020}. 
Recently, no-gap second-order optimality conditions for optimization problems with non-uniformly convex integrands were considered in \cite{WachsmuthWachsmuth2022}, where the integral functional is extended to the space of measures and the obtained second-order derivatives contain integrals on $(N-1)$-dimensional manifolds. The results derived in \cite{WachsmuthWachsmuth2022} were applied to bang-off-bang control problems governed by smooth semilinear elliptic PDEs. 
	The appearance of integrals on $(N-1)$-dimensional manifolds in the expression of second-order generalized derivatives was first observed in \cite{ChristofMeyer2019}, 
	where the authors studied the stability of solutions to $H^1_0$-elliptic variational inequalities of the second kind that include a non-differentiable Nemytskii operator.

\medskip 

Regarding optimal control of non-smooth PDEs without 
the $L^1$-norm
of the control in the objective, there are comparatively few contributions. 
On the topic related to first-order optimality conditions, we refer to \cite{Constantin2018} for a semilinear elliptic equation,  to \cite{ClasonNhuRosch} for a class of quasilinear elliptic equations, and to \cite{MeyerSusu2017} for semilinear parabolic equations.
For the second-order necessary and sufficient optimality conditions, 
as far as we know a few contributions were addressed by Clason et al.  \cite{ClasonNhuRosch2020} for quasilinear elliptic equations and by Nhu \cite{Nhu2021Optimization} for the specific case \eqref{eq:max-pde}.
In these papers, in order to derive the second-order optimality conditions, a curvature functional of the objective was introduced and this functional is considered as a generalized second-order derivative of the cost functional. Note that, in these contributions, 
no explicit formulas for the curvature functional have
been established.

\bigskip
The main aim of this paper is to derive the first-order optimality conditions, explicit second-order (necessary and sufficient) optimality conditions, and the stability of local minimizers of \eqref{eq:P}; see \cref{thm:1st-OS,thm:2nd-OS-nec-explicit,thm:2nd-OS-suf-explicit,thm:stability-Lipschitz} below. To achieve this goal, we first obtain $C$-stationarity conditions by following a regularization scheme \cite{Barbu1984}; see, also \cite{Constantin2018,MeyerSusu2017,ClasonNhuRosch}, where the associated adjoint state and multipliers are proven to be uniquely determined for any local optimal control under the assumption that the level sets at $\tau_i$, $1 \leq i \leq K$, of the considered optimal state have measure zero, see \cref{thm:1st-OS} and \cref{cor:projections} below.
In order to derive the second-order optimality conditions, we follow \cite{ClasonNhuRosch2020} to introduce the curvature functional $Q$ of the smooth component $F$ 
in the objective and then combine with the technique in \cite{CasasHerzogWachsmuth2012}.
Under the following  assumption on the structure of 
	the mentioned state
$\bar y$ (depending on the parameter $\kappa$):
\begin{equation} 
	\sum_{i =1}^K\lambda^N\left( \{ |\bar y - \tau_i| < \epsilon \} \right) \leq c_s \epsilon \quad \text{for all } \epsilon \in (0,1) \nonumber
\end{equation}
with some constant $c_s >0$, \cref{prop:Q-weak2subder} below shows that the curvature functional is identical to the strong second subderivative of $F$ introduced in \cite[Def.~3.2]{WachsmuthWachsmuth2022}. 
(Here $\lambda^d$ stands for the $d$-dimensional Lebesgue measure and when $d=N$, we simply write $dx := d\lambda^N(x)$.)
A second-order Taylor type expansion then gives the second-order conditions for \eqref{eq:P} after proving 
	some key technical limits.
Finally, under a more restrictive  condition on the structure of $\bar y$ as follows
\begin{equation*}
	\bar y(x) = \tau_i \quad \implies \quad |\nabla \bar y(x)| \neq 0, \quad  1 \leq i \leq K,
\end{equation*}
an explicit formulation of the curvature functional $Q$ is established and contains the integral on $(N-1)$-dimensional manifolds (see \cref{thm:Q-explicit-form}). Consequently, we derive, respectively, the corresponding second-order necessary and sufficient optimality conditions in the explicit form as follows:
\begin{equation} \nonumber
	Q(\bar u, \bar p; h) \geq 0 \, 
	{
	\forall h\in \mathcal{C}_{L^\barr(\Omega)}({U}_{ad};\bar u)
}%
	\quad \text{and} \quad Q(\bar u, \bar p; h) > 0 \, 
	{
	\forall h\in \mathcal{C}_{L^2(\Omega)}({U}_{ad};\bar u)\setminus \{0\}
}%
\end{equation} 
	for some fixed number $\barr > \frac{N}{2}$ and $\barr \geq 2$, where
\begin{multline} \label{eq:expclit-Q}
	Q(\bar u, \bar p; h) =\int_\Omega  \frac{\partial^2 L}{\partial y^2}(x, \bar y) (S'(\bar u) h)^2 + \nu h^2 - \1_{\{ \bar y \notin E_f \}} \bar p f''_{yy}( \bar y)(S'(\bar u) h)^2 \dx \\
	 + \sum_{i=1}^K \sigma_i \int_{ \{ \bar y = \tau_i \} } \frac{\bar p (S'(\bar u)h)^2}{|\nabla \bar y|} \dH^{N-1}.   
\end{multline} 
Here $\bar p$ is an adjoint state associated with the control $\bar u$;
$\{ \bar y = \tau_i \} := \{x \in \overline\Omega \mid \bar y(x) = \tau_i \}$; 
$E_f$ is the set of non-smooth points of $f$ and determined in \eqref{eq:non-smooth-points} below;
$\sigma_i$, defined in \eqref{eq:sigma-i}, stands for the  jump of the derivative of $f$ in the singular point $\tau_i$;
$\mathcal{C}_{L^\barr(\Omega)}({U}_{ad};\bar u)$ (depending on $\kappa$) is a cone of critical directions defined as in \eqref{eq:critical-cone};
$S$ is the control-to-state operator determined later in \cref{sec:standing-assumption};
and 
$\Ha^d$ is the $d$-dimensional Hausdorff measure on $\Rbb^N$ that is scaled as in \cite[Def.~2.1]{Evans1992}.
These explicit formulas 
could play a significant role in showing the error estimates for the numerical approximation of \eqref{eq:P}; see, e.g. \cite{ClasonNhuRosch2022,ClasonNhuRosch2023_part2}. Such error estimates will be studied in a follow-up research work. 
	The explicit formulation \eqref{eq:expclit-Q} for the curvature functional $Q$ is an improvement in comparison with the results obtained in \cite{ClasonNhuRosch2020,Nhu2021Optimization,ClasonNhuRosch2022}. 
	There, \cite{ClasonNhuRosch2020,ClasonNhuRosch2022} and \cite{Nhu2021Optimization} considered, respectively, the curvature functional for quasilinear and semilinear elliptic problems. 	
	In \cite{ClasonNhuRosch2020,Nhu2021Optimization}, only the general definition of $Q$ was given  and its explicit formulation was not established.
	In \cite{ClasonNhuRosch2022}, 
	an explicit formulation for the curvature functional in the two-dimensional setting was stated 
 	by exploiting an analysis of level sets, where the domain $\Omega$ is required to be a bounded convex polygonal. 
 	In this paper, we shall formulate \eqref{eq:expclit-Q} in the $N$-dimensional setting with $N$ being any nonzero integer
 	and in the case where the domain has a Lipschitz boundary only.
 	In order to do this, in contrast to the approach employed in \cite{ClasonNhuRosch2022}, we will follow the technique from \cite{ChristofMeyer2019} (also \cite{WachsmuthWachsmuth2022} and \cite{ChristofWachsmuth2018}). We first prove \eqref{eq:expclit-Q} in the lower dimensional space and then validate it in a higher dimensional setting via using a standard partition-of-unity argument.
Finally, for the study of stability as $\kappa \to 0$, we show in \cref{thm:stability-Lipschitz} that, under a quadratic growth condition for a local minimizer of $(P_0)$ ($\kappa = 0$), the optimal solutions to \eqref{eq:P} converge strongly and the rate of convergence is $O(\kappa)$.   
This result on the Lipschitz stability with respect to (w.r.t.) the sparsity parameter is new.

\medskip 

Finally, let us emphasize that our approach 
	showing the second-order optimality conditions
is different from the one in \cite{WachsmuthWachsmuth2022} (see, also \cite{ChristofWachsmuth2018}). There, the authors first considered  general optimization problems with a non-uniformly convex and non-smooth integral functional and they then applied the established results to bang-off-bang optimal control problems with smooth elliptic PDEs to obtain the no-gap second-order optimality conditions, which are equivalent to a quadratic growth condition in 
\textcolor{blue}{
the $L^1$-norm.
}%
In order to do that, a structural assumption imposing on the adjoint state is required; see, Assumption 4.12 in \cite{WachsmuthWachsmuth2022} (and Assumption 6.6 in \cite{ChristofWachsmuth2018}). This condition is essential and plays a significant contribution in proving a so-called \emph{nondegeneracy condition} (see \cite[Lems~2.12 \& 4.9]{WachsmuthWachsmuth2022}). Our approach, by contrast, do not need any structural assumption on the adjoint state.

\bigskip

The organization of the remainder of this paper is as follows. In \cref{sec:standing-assumption}, we introduce the notation,  standing assumptions that shall be employed throughout the whole article, and some preliminary results on the state equation. 
In \cref{sec:1stOS}, the first-order optimality conditions are established, 
in which 
in addition to the existence of the corresponding adjoint state, two other  multipliers exist. One of them comes from the subdifferential in the sense of convex analysis of the $L^1$-norm functional $j$ and the other pointwise belongs to Clarke’s subdifferential of the non-smooth coefficient $f$. 
\cref{sec:2ndOS} is devoted to the second-order optimality conditions for \eqref{eq:P}. The definition and some required properties of the curvature functional of $F$ 
are stated in \cref{sec:2nd-OC:curvature}. 
Second-order conditions that are based on the curvature functional are presented in \cref{sec:2ndOS-based-Q}, while the explicit forms of these conditions  are provided in \cref{sec:2ndOS-explicit-form}. Finally, the stability of optimal solutions is stated in \cref{sec:stability}.

\section{Notation, standing assumptions and preliminary results} \label{sec:standing-assumption}

\paragraph*{Notation}
For a given point $u$ in a Banach space $X$ and a number $\rho>0$, symbols $B_X(u,\rho)$ and $\overline B_X(u,\rho)$ denote the open and closed balls of radius $\rho$ centered at $u$, respectively. 
For Banach spaces $X$ and $Y$, symbols $X \hookrightarrow Y$ and $X \Subset Y$ stand, respectively, for the canonical continuous embedding and the canonical compact embedding of $X$ in $Y$; the notation $\Linop(X,Y)$ denotes the Banach space of all linear continuous operators from $X$ to $Y$. 
By $c_0^+$, we denote the set of all positive sequences that converge to zero. For any function $y=y(x)$, symbols $y^+$ and $y^{-}$ stand, respectively, for the positive and negative parts of $y$, that is, $y^+(x) := \max\{0, y(x)\}$ and $y^{-}(x) := \max\{0, -y(x) \}$ for all $x$. 
The symbol $\partial_C f$ 
	denotes the Clarke subdifferential 
of a locally Lipschitz continuous function $f$, while for a convex function $g$ the notation $\partial g$ denotes the subdifferential in the sense of convex analysis.
For a function $g:\Omega\to \Rbb$ and a subset $M \subset \Rbb$, 
 	we put $\{ g \in M \} := \{ x \in \Omega \mid g(x) \in M\}$.
  	Analogously, for given functions $g_1,g_2$ and subsets $M_1, M_2 \subset \Rbb$, we set
  	$\{ g_1 \in M_1, g_2 \in M_2 \} := \{ g_1 \in M_1\} \cap \{g_2 \in M_2 \}$.
Finally, by $C$ we denote a generic positive constant, which might be different at different places of occurrence. We also write, e.g., $C_\xi$ for a constant dependent only on the parameter $\xi$.

\medskip

Throughout the paper, we need the following standing assumptions.
\begin{assumption} \label{ass:standing}
	\item \label{ass:domain}
		$\Omega \subset \Rbb^N$, 
		{
		$N \geq 1$, 
	}%
		is a bounded domain with a Lipschitz boundary $\partial \Omega$. 
	\item \label{ass:Tikhonov-para}
		There hold $\nu > 0$, $\kappa \geq 0$, and $- \infty < \alpha < 0 < \beta < \infty$.
		
	\item \label{ass:nonsmooth-coefficient}
		{
		Function $f$ is not differentiable at $\tau_i$, $1 \leq i \leq K$, that is,
		\begin{equation}
			\label{eq:sigma-i}
			\sigma_i := f'_{i-1}(\tau_i) - f'_{i}(\tau_i) \neq 0 \quad \text{for all} \quad 1 \leq i \leq K.
		\end{equation}	
	}%
	\item \label{ass:ellipticity}
		The coefficients of $A$ admit the following regularity properties: $a_0 \in L^\infty(\Omega)$, $a_0(x) \geq 0$ for a.a. $x \in \Omega$, $a_{ij} \in L^\infty(\Omega)$, and 
		\begin{equation} \nonumber
			\sum_{i,j=1}^N a_{ij}(x) \xi_i\xi_j \geq \Lambda |\xi|^2 \quad \text{for a.a. } x \in \Omega, \, \text{for all } \xi \in \Rbb^N, \, \text{and for some } \Lambda >0.
		\end{equation} 	
	\item \label{ass:integrand}
	The function $L: \Omega \times \Rbb \to \Rbb$ is Carath\'{e}odory of class $C^2$ w.r.t. the second variable  such that $L(\cdot,0) \in L^1(\Omega)$. Moreover, for any $M>0$, functions $\phi_M \in L^2(\Omega)$, $\psi_M \in L^1(\Omega)$ exist and satisfy
	\begin{align*}
		&| \frac{\partial L}{\partial y} (x,y)| \leq \phi_M(x), \quad |\frac{\partial^2 L}{\partial y^2}(x,y)| \leq \psi_M(x), \\
		&|\frac{\partial^2 L}{\partial y^2}(x,y_1) -\frac{\partial^2 L}{\partial y^2}(x,y_2) | \leq \psi_M(x)|y_1 - y_2|
	\end{align*}
	for a.a. $x \in \Omega$ and for all $y, y_1, y_2 \in \Rbb$ satisfying $|y|, |y_1|, |y_2| \leq M$.

\end{assumption}

A standard for the choice of $L$ is the quadratic function
\begin{equation*}
	L(x,y) = \frac{1}{2}(y- y_d(x))^2 \quad \text{with $y_d \in L^2(\Omega)$.}
\end{equation*}

\medskip
We now consider the state equation \eqref{eq:state}.
Let $\barr > \frac{N}{2}$ be arbitrary, but fixed. 
From \cite[Thm.~4.7]{Troltzsch2010}, for each $u\in L^\barr(\Omega)$, there exists a unique solution $y_u \in H^1_0(\Omega) \cap C(\overline\Omega)$ of \eqref{eq:state}. By $S$ we now denote the solution operator of \eqref{eq:state} as a function from $L^\barr(\Omega)$ to $H^1_0(\Omega) \cap C(\overline\Omega)$. Since $L^\barr(\Omega) \Subset W^{-1,p}(\Omega)$ for some $p >N$, $S$ is weakly-to-strongly continuous; see, e.g., \cite[Thm.~2.2]{CasasTroltzsch2009}.

\medskip 

We now denote by $E_f$ the exceptional set of all non-differentiability points of $f$, that is,
\begin{equation}
	\label{eq:non-smooth-points}
	E_{f} := \left\{ \tau_1, \tau_2,\ldots,\tau_\barK\right\}.
\end{equation}
Obviously, $f$ is directionally differentiable and its directional derivative is given by
\begin{equation}
	\label{eq:f-dir-der}
	f'(t; h) =  \left\{
	\begin{aligned}
		& \sum_{i=0}^{\barK}  \1_{(\tau_{i},\tau_{i+1})}(t)f_i'(t)h && \text{if } t \notin E_f,\\
		& \1_{(0, \infty)}(h) f_{i}'(\tau_{i})h + \1_{(-\infty, 0)}(h) f_{i-1}'(\tau_{i})h && \text{if } t = \tau_i, 1 \leq i \leq K 
	\end{aligned}
	\right.
\end{equation}
with noting that $\tau_0 = - \infty$ and $\tau_{K+1} = +\infty$. 

\medskip
From now on, let $\barr$ be an arbitrary, but fixed number satisfying
\begin{equation}
	\label{eq:barr-choice}
	\barr > \frac{N}{2} \quad \text{and} \quad \barr \geq 2.
\end{equation}
The next proposition states the Lipschitz continuity and the directional differentiability of the control-to-state operator $S$. 
	Its proof is similar to that of
\cite[Prop.~3.1]{Nhu2021Optimization} and  thus omitted here.

\begin{proposition} \label{prop:control-to-state}
	The control-to-state mapping $S: L^\barr(\Omega) \to H^1_0(\Omega) \cap C(\overline\Omega)$, $u \mapsto y$, associated with the state equation \eqref{eq:state} satisfies the following assertions:
	\begin{enumerate}[label=(\roman*)]
		\item \label{it:Lipschitz-app} 
		$S$ is globally Lipschitz continuous; 
		\item \label{it:dir-der-app} $S$ is Hadamard directionally differentiable at any $u \in L^\barr(\Omega)$ in any direction $h \in L^\barr(\Omega)$. Moreover, for any $u,h \in L^\barr(\Omega)$, a $S'(u;h) \in H^1_0(\Omega) \cap C(\overline\Omega)$ exists and satisfies
		\begin{equation} \label{eq:weak-strong-Hadamard-app}
			\frac{S(u + t_k h_k) - S(u)}{t_k} \to S'(u;h) \quad \text{in } H^1_0(\Omega) \cap C(\overline\Omega) 
		\end{equation}
		when $h_k \rightharpoonup h$ in $L^\barr(\Omega)$ and $t_k \to 0^+$.
		Furthermore, $\delta_h := S'(u;h)$ uniquely solves 
		\begin{equation}
			\label{eq:dir-der-app}
			A \delta_h + f'(S(u);\delta_h)  = h \, \text{in } \Omega, \quad \delta_h =0 \, \text{on } \partial\Omega.
		\end{equation}
		Here with a little abuse of notation, we denote by $f'(S(u);\delta_h)$ the superposition operator;
		\item \label{it:max-pri-app} $S'(u;h)$ fulfills the maximum principle, i.e., 
		\[
			h \geq 0 \, \text{a.a. in } \Omega  \, \implies \, S'(u;h) \geq 0 \, \text{a.a. in } \Omega;
		\]
		\item \label{it:wlsc-app} $S'(u;\cdot)$ is weakly-to-strongly continuous.
	\end{enumerate} 
\end{proposition}
Next, we define a solution operator of linear elliptic PDEs. 
\begin{definition} \label{def:sol-oper-adjoint}
	Given a function $\chi \in L^\infty(\Omega)$ with $\chi \geq 0$ a.a. in $\Omega$, the operator $G_\chi \in \Linop(L^\barr(\Omega), H^1_0(\Omega) \cap C(\overline\Omega) )$  maps any $h \in L^\barr(\Omega)$ to $z =: G_\chi h  \in H^1_0(\Omega) \cap C(\overline\Omega)$, where $z$ uniquely solves the equation $A z  + \chi z  = h$ in $\Omega$, $z =0$ on $\partial\Omega$.
\end{definition}
From now on, let us fix $\epsilon_0$ such that
\begin{equation}
	\label{eq:epsilon-zero}
	0 < \epsilon_0 < \frac{\tau_{i+1}- \tau_i}{2} \quad \text{for all} \quad 1 \leq i \leq K -1.
\end{equation}

The following proposition presents a precise characterization of points at which the control-to-state operator $S$ is G\^{a}teaux-differentiable. 
\begin{proposition}[{cf. \cite[Cor.~2.3]{Constantin2018}}]
	\label{prop:G-diff-control2state}
	The control-to-state operator $S: L^\barr(\Omega) \to H^1_0(\Omega) \cap C(\overline\Omega)$ is G\^{a}teaux-differentiable 
	at $u$ 
	if and only if 
	\begin{equation}
		\label{eq:zero-measure}
		\lambda^N(\{ S(u) \in E_f\}) =0.
	\end{equation}
	Moreover, under the condition \eqref{eq:zero-measure}, the directional derivative $\delta_h := S'(u;h)$ in a direction $h \in L^\barr(\Omega)$ is determined as 
	$\delta_h = G_{\chi}h$ with $\chi(x) = \1_{\{y_u \notin E_f \}}(x) f'(y_u(x))$ for a.a. $x \in \Omega$ and for $y_u :=S(u)$.	
\end{proposition}
\begin{proof}
	Assume that \eqref{eq:zero-measure} is fulfilled. For any $h \in L^\barr(\Omega)$, $\delta_h := S'(u;h)$ satisfies \eqref{eq:dir-der-app}. According to \eqref{eq:f-dir-der}, there holds
	\begin{align*}
		f'(y_u;\delta_h) &=  \sum_{i=0}^{\barK}  \1_{(\tau_{i},\tau_{i+1})}(y_u)f_i'(y_u)\delta_h \\
		& \qquad  +\sum_{i=1}^{\barK}  \1_{\{\tau_{i}\}}(y_u) \left[ \1_{(0, \infty)}(\delta_h) f_{i}'(\tau_{i})\delta_h + \1_{(-\infty, 0)}(\delta_h) f_{i-1}'(\tau_{i})\delta_h\right]\\
		& = \sum_{i=0}^{\barK}  \1_{(\tau_{i},\tau_{i+1})}(y_u)f_i'(y_u)\delta_h = f'(y_u)\delta_h 
	\end{align*}
	a.a. in $\Omega$.
	Here we have used the condition \eqref{eq:zero-measure} to have that $\1_{\{\tau_{i}\}}(y_u) =0$ a.a. in $\Omega$ and for all $1 \leq i \leq K$. Therefore $f'(y_u;\cdot)$ is linear in the second variable and so is $S'(u;\cdot)$. We have thus shown the G\^{a}teaux-differentiability of $S$ in $u$. Moreover, we obviously have $\delta_h = G_{\chi} h$ with $\chi(x) = \1_{\{y_u \notin E_f \}}(x) f'(y_u(x))$ for a.a. $x \in \Omega$.	
	
	Conversely, we now prove the condition \eqref{eq:zero-measure} under the G\^{a}teaux-differentiability of $S$ in $u$. To this end, take any $h \in L^\barr(\Omega)$ and set $\delta_h := S'(u;h)$ and $\delta_{-h} := S'(u;-h)$. 
	Thanks to the directional derivative formula \eqref{eq:f-dir-der}, we have
	\begin{multline*}
		f'(y_u;\delta_h) + f'(y_u; \delta_{-h})  = \sum_{i=0}^{\barK}  \1_{(\tau_{i},\tau_{i+1})}(y_u)f_i'(y_u)(\delta_h +\delta_{-h}) \\
		\begin{aligned}
		  & +\sum_{i=1}^{\barK}  \1_{\{\tau_{i}\}}(y_u) \left\{  f_{i}'(\tau_{i}) [\1_{(0, \infty)}(\delta_h)\delta_h + \1_{(0, \infty)}(\delta_{-h}) \delta_{-h}] \right. \\
		  & \left. + f_{i-1}'(\tau_{i})[\1_{(-\infty, 0)}(\delta_h)  \delta_h + \1_{(-\infty, 0)}(\delta_{-h}) \delta_{-h}] \right\}.
		\end{aligned}
	\end{multline*}
	Since $\delta_h + \delta_{-h} =0$, there holds
	\begin{multline*}
		f'(y_u;\delta_h) + f'(y_u; \delta_{-h}) = \sum_{i=1}^{\barK}  \1_{\{\tau_{i}\}}(y_u) \delta_h \left\{ f_{i}'(\tau_{i}) \left[ \1_{(0, \infty)}(\delta_h) - \1_{(0, \infty)}( -\delta_{h}) \right] \right. \\
		\begin{aligned}
			& \qquad \left. + f_{i-1}'(\tau_{i})\left[ \1_{(-\infty, 0)}(\delta_h) - \1_{(-\infty, 0)}(-\delta_{h})\right] \right\} \\
			& = \sum_{i=1}^{\barK}  \1_{\{\tau_{i}\}}(y_u) |\delta_h| [f_{i}'(\tau_{i}) - f_{i-1}'(\tau_{i}) ]  = -\sum_{i=1}^{\barK}  \1_{\{\tau_{i}\}}(y_u) |\delta_h|  \sigma_i.
		\end{aligned}
	\end{multline*}
	Here, we have just used the definition \eqref{eq:sigma-i} of $\sigma_i$ to derive the last identity.
	Moreover, adding the equations for $\delta_h$ and $\delta_{-h}$ yields $f'(y_u;\delta_h) + f'(y_u; \delta_{-h})  =0$. This implies that
	\begin{equation} \label{eq:vanish}
		\sum_{i=1}^{\barK}  \1_{\{\tau_{i}\}}(y_u) |\delta_h|  \sigma_i = 0.
	\end{equation}
	Obviously, the sets $\{ |y_u - \tau_i| < \epsilon_0 \} \backslash \partial\Omega$, $1 \leq i \leq K$, are open in $\Omega$ and mutually disjoint. Here $\epsilon_0$ is defined as in \eqref{eq:epsilon-zero}. We now fix $1 \leq i_0 \leq K$ and choose $\varphi_{i_0} \in C^\infty(\Rbb^N)$ such that
	\begin{equation} \nonumber
		\varphi_{i_0}  >0 \, \text{on } \{ |y_u - \tau_{i_0}| < \epsilon_0 \} \backslash \partial\Omega \quad \text{and} \quad \varphi_{i_0} = 0 \, \text{on } \Rbb^N \backslash( \{ |y_u - \tau_{i_0}| < \epsilon_0 \} \backslash \partial\Omega);
	\end{equation} 
	see, e.g. \cite[Lem.~B.1]{Constantin2018}. 
	Setting $h := A\varphi_{i_0} + f'(y_u;\varphi_{i_0}) \in L^\barr(\Omega)$ and using the uniqueness of solutions to \eqref{eq:dir-der-app}, we have $\delta_h = \varphi_{i_0}$.
	Combining this with \eqref{eq:vanish} 
	gives
	$
		\1_{\{\tau_{i_0}\}}(y_u) \varphi_{i_0}  \sigma_{i_0} = 0,
	$
	and thus $\lambda^N(\{ y_u = \tau_{i_0} \}) =0$,
	due to \cref{ass:nonsmooth-coefficient}.	
	Similarly, $\lambda^N(\{ y_u = \tau_{i} \}) =0$ for any $1 \leq i \leq K$. We thus have \eqref{eq:zero-measure}.	
\end{proof}

The following corollary is a direct consequence of \cref{prop:G-diff-control2state} and a chain rule. Its proof is elementary and thus is skipped.
\begin{corollary}
	\label{cor:G-diff}
	Let $u \in L^\barr(\Omega)$ satisfy \eqref{eq:zero-measure}. 
	Then $F$ is G\^{a}teaux-differentiable in $u$. Moreover, for any $h_n \rightharpoonup h$ in $L^\barr(\Omega)$, there holds 
	\begin{equation} \nonumber
		\lim\limits_{n \to \infty } \frac{F(u + t_nh_n) - F(u)}{t_n} = F'(u)h = \int_\Omega (p_u + \nu u) h\dx 
	\end{equation} 
	with $p_u := G_{\chi}^*(\frac{\partial L}{\partial y}(\cdot, y_u) )$, where $\chi(x) := \1_{\{y_u \notin E_f \}}(x) f'(y_u(x))$ for a.a. $x \in \Omega$.
\end{corollary}

\section{First-order optimality conditions} \label{sec:1stOS}
We first address the existence of minimizers. By $U_{ad}$, 
	we denote the set of admissible points 
of \eqref{eq:P}, that is,
\begin{equation} \nonumber
	U_{ad} : = \left\{ u \in L^\infty(\Omega) \mid \alpha \leq u(x) \leq \beta \, \text{for a.a. } x \in \Omega  \right\}.
\end{equation} 
Obviously, $U_{ad}$ is a closed, convex and bounded set in $L^\infty(\Omega)$ and thus weakly closed in $L^\barr(\Omega)$. From this and the weak lower semicontinuity of the objective $J_\kappa(u)$, we deduce from a standard argument the following result on existence of optimal controls for \eqref{eq:P}. 
\begin{proposition}
	\label{prop:existence}
	There exists a local minimizer $\bar u_\kappa \in U_{ad}$ of \eqref{eq:P}.
\end{proposition}
	To simplify the notation, 
from now on except \cref{sec:stability} we will drop the subscript $\kappa$ in the symbol for any minimizer $\bar u_\kappa$ of \eqref{eq:P}, that is, we shall write $\bar u$ as a local minimizer defined as in \cref{prop:existence}.


\medskip 

In the study of the first- and second-order optimality conditions, in addition to the lack of the differentiability of the control-to-state operator, there is an obstacle coming from the fact that the functional $j(u) = \norm{u}_{L^1(\Omega)}$ is only directionally differentiable. Its directional derivative is defined as
\begin{equation}
	\label{eq:j-dir-der}
	j'(u, v) := \lim\limits_{t \to 0^+} \frac{j(u+tv) - j(u)}{t} = \int_{\{ u>0 \} } v\dx - \int_{\{ u < 0 \} } v\dx + \int_{\{ u =0 \} } |v|\dx
\end{equation}
for $u, v \in L^1(\Omega).$
Moreover, since $j$ is Lipschitz continuous and convex, the subdifferential in the  sense of convex analysis coincides with the generalized gradients introduced by Clarke \cite{Clarke2}. By a simple computation, we derive
\begin{equation}
	\label{eq:subderivative-j-formulation}
	\lambda \in \partial j(u) \quad \Longleftrightarrow \quad
	\lambda(x) \left\{ \begin{aligned}
	&=1 && \text{ for a.a. } x \in \{ u>0\},\\
	&=-1 &&\text{ for a.a. } x \in \{ u<0\},\\
	&\in [-1,1] && \text{ for a.a. } x \in \{ u=0\}.
	\end{aligned}
	\right.
\end{equation}
Furthermore, one has
\begin{equation}
	\label{eq:subdiff-directional-relation}
	\max \limits_{\lambda \in \partial j(u)} \int_\Omega \lambda(x) v(x) \dx = j'(u;v) \leq  \frac{j(u+tv) - j(u)}{t} \quad \forall 0 < t <1.
\end{equation}
We refer to  \cite{CasasHerzogWachsmuth2012,CasasTroltzsch2020,WachsmuthWachsmuth2011,CasasTroltzsch2019,Casas2012,CasasTroltzsch2014} for these facts on the subdifferential of $j$. Finally, for $v_n \rightharpoonup v$ in $L^1(\Omega)$ and $\{t_n\} \in c_0^+$, there holds
\begin{equation} \label{eq:j-weak-dir-deri}
	\liminf\limits_{n \to \infty} \frac{j(u+t_nv_n) - j(u)}{t_n} \geq j'(u; v).
\end{equation}
To show this, by a simple computation, we can write
\begin{multline*}
	\frac{j(u+t_nv_n) - j(u)}{t_n} = \int_{\{ u>0 \} } \frac{|u + t_n v_n| - u}{t_n} \dx + \int_{\{ u < 0 \} } \frac{|u + t_n v_n| + u}{t_n} \dx \\
	\begin{aligned}
		& + \int_{\{ u =0 \} } |v_n| \dx 
		 \geq \int_{\{ u>0 \} } v_n  \dx + \int_{\{ u < 0 \} } (-v_n) \dx + \int_{\{ u =0 \} } |v_n| \dx. 
	\end{aligned}
\end{multline*}
Thanks to the weak lower semicontinuity of 
the $L^1$-norm, 
one has
\begin{equation} \nonumber
	\liminf\limits_{n \to \infty}\int_{\{ u =0 \} } |v_n| \dx  \geq \int_\Omega \1_{\{ u =0 \} } |v|\dx = \int_{\{ u =0 \} } |v|\dx,
\end{equation} 
which infers
\begin{align*}
	\liminf\limits_{n \to \infty}\frac{j(u+t_nv_n) - j(u)}{t_n} & \geq \int_{\{ u>0 \} } v\dx - \int_{\{ u < 0 \} } v\dx + \int_{\{ u =0 \} } |v|\dx.
\end{align*}
This together with \eqref{eq:j-dir-der} leads  to \eqref{eq:j-weak-dir-deri}. As seen in the proof of \cref{thm:2nd-OS-suf}, the estimate \eqref{eq:j-weak-dir-deri} plays an important role in showing that a direction $h$ belongs to a critical cone defined in \eqref{eq:critical-cone} below when $h$ is a weak accumulation point of an appropriate sequence generated by applying a contradiction argument.

\medskip 
To derive C-stationarity conditions, we apply the adapted penalization method of Barbu \cite{Barbu1984}.
The state equation shall be regularized via a classical mollification of the non-smooth nonlinearity.
Let $\psi \in C^\infty_0(\Rbb)$ be a non-negative function such that $\mathrm{supp}(\psi) \subset [-1,1]$, $\int_{\Rbb} \psi(\tau)d\tau = 1$
and define the family $\{f_\epsilon\}_{\epsilon>0}$ of functions
\begin{equation} \notag
	f_\epsilon := \frac{1}{\epsilon}f * \left(\psi\circ (\epsilon^{-1}{\Id})\right)
\end{equation}
with $f *g$ standing for the convolution of $f$ and $g$, and $\Id$ being the identity mapping.
Then, $f_\epsilon \in C^\infty(\Rbb)$ by a standard result.
In addition, a simple calculation shows that
\begin{equation*} 
	f'_\epsilon(t) \geq 0 \quad \forall t \in \Rbb.
\end{equation*}
Moreover, for any $M>0$, there exists a constant $C_M >0$ independent of $\epsilon$ satisfying for all $\epsilon \in (0,1)$ that
\begin{align} \label{eq:f_regu}
	&|f_\epsilon(t) - f(t) | \leq C_{M}\epsilon && \text{for all} \quad t \in \Rbb, |t|\leq M, \\
	& |f_\epsilon (t_1) - f_\epsilon(t_2) | \leq C_{M} |t_1 - t_2 | && \text{for all} \quad t_i \in \Rbb, |t_i | \leq M, i=1,2. \label{eq:f_regu_Lip}
\end{align}
We now consider the regularized state equation associated with \eqref{eq:state}
\begin{equation} \label{eq:state-regularized}
	A y + f_\epsilon(y)  = u \, \text{in } \Omega, \quad
	y =0 \, \text{on } \partial\Omega \quad \text{for } u \in L^\barr(\Omega).
\end{equation}
By $S_\epsilon$ we denote the regularization of the control-to-state operator $S$, that is, $S_\epsilon$ maps any $u \in L^\barr(\Omega)$ to the corresponding $y \in H^1_0(\Omega) \cap C(\overline\Omega)$ uniquely solving \eqref{eq:state-regularized}. 
Some useful properties of $S_\epsilon$ are listed below.
\begin{theorem} \label{thm:S-reg}
	The following assertions are fulfilled:
	\begin{enumerate}[label = (\roman*)]
		\item \label{item:S-reg-deri} 
		$S_\epsilon$ is of class $C^1$. Furthermore, for any $u, h \in L^\barr(\Omega)$, $z:= S'_\epsilon(u)h$ is the unique solution in $H^1_0(\Omega) \cap C(\overline\Omega)$ of the Dirichlet problem
		\begin{equation*} 
			A z + f'_\epsilon(S_\epsilon(u))z  = h \, \text{in } \Omega, \quad
			z =0 \, \text{on } \partial\Omega;
		\end{equation*}
		\item \label{item:w2s-converge}
		 If $u_n \rightharpoonup u$ in $L^\barr(\Omega)$ as $n \to \infty$, then $S_\epsilon(u_n) \to S_\epsilon (u)$ in $H^1_0(\Omega) \cap C(\overline\Omega)$ as $n \to \infty$;
		 \item \label{item:S-regularization} 
		 For any bounded set $U$ in $L^\barr(\Omega)$, there exists a constant $C_U>0$ such that
		 \begin{equation} \nonumber
		 	\norm{S_\epsilon(u) - S(u) }_{H^1_0(\Omega)} + \norm{S_\epsilon(u) - S(u)}_{C(\overline\Omega)} \leq C_U \epsilon \quad \text{for all } u \in U;
		 \end{equation} 
		 \item \label{item:w2s-converge-main}
		 If $u_\epsilon \rightharpoonup u$ in $L^\barr(\Omega)$ as $\epsilon \to 0^+$, then $S_\epsilon(u_\epsilon) \to S(u)$ in $H^1_0(\Omega) \cap C(\overline\Omega)$ as $\epsilon \to 0^+$.
	\end{enumerate}
\end{theorem}
\begin{proof}
	The arguments for \ref{item:S-reg-deri} and \ref{item:w2s-converge} are standard; see, e.g. \cite[Thms.~2.1 \& 2.3]{CasasTroltzsch2020}. For \ref{item:S-regularization}, by taking any bounded set $U \subset L^\barr(\Omega)$, there exists a constant $M = M(U) >0$ satisfying
	\begin{equation} \nonumber
		\norm{S(u)}_{C(\overline\Omega)}, \norm{S_\epsilon(u)}_{C(\overline\Omega)} \leq M \quad \text{for all } u \in U, \epsilon \in (0,1).
	\end{equation} 
	Taking $u \in U$ arbitrarily, setting $y := S(u)$, $y_\epsilon := S_\epsilon(u)$, and subtracting the equations for $y_\epsilon$ and $y$ yield
	\begin{equation*}
		\label{eq:subtracting-y}
		A (y_\epsilon - y) + (f(y_\epsilon) - f(y))   = f(y_\epsilon) - f_\epsilon(y_\epsilon) \, \text{in } \Omega, \quad
		y_\epsilon - y =0 \ \text{on } \partial\Omega.
	\end{equation*}
	Testing the above equation by $y_\epsilon - y$, applying the monotonicity of $f$, and employing the strong ellipticity condition in \cref{ass:ellipticity}, we deduce from a Cauchy–Schwarz inequality that
	\begin{multline*}
		\Lambda \norm{\nabla (y_\epsilon - y)}_{L^2(\Omega)}^2  \leq \int_\Omega [f(y_\epsilon) - f_\epsilon(y_\epsilon)](y_\epsilon - y) \dx \\
		\leq \norm{ y_\epsilon - y}_{L^2(\Omega)} \norm{f(y_\epsilon) - f_\epsilon(y_\epsilon)}_{L^2(\Omega)} \leq  C_M\epsilon \lambda^N(\Omega)^{1/2}\norm{ y_\epsilon - y}_{L^2(\Omega)},
	\end{multline*}
	where we have just used  \eqref{eq:f_regu} and the fact that $\norm{y_\epsilon}_{C(\overline\Omega)}, \norm{y}_{C(\overline\Omega)} \leq M$ to get the last inequality. This and the Poincar\'{e} inequality imply that
	$
		\norm{y_\epsilon - y}_{H^1_0(\Omega)} \leq C \epsilon.
	$
	We now prove that 
	\begin{equation}
		\label{eq:Linfty-esti}
		\norm{y_\epsilon - y}_{L^\infty(\Omega)} \leq C \epsilon
	\end{equation}
	and \ref{item:S-regularization} then follows. To this aim, we put
	$
		a_\epsilon  := \1_{\{ y_\epsilon \neq y \}} \frac{f(y_\epsilon) - f(y)}{y_\epsilon - y}
	$
	and have $0 \leq a_\epsilon \leq C_M$ a.a. in $\Omega$ thanks to the monotonicity and the local Lipschitz continuity of $f$. Obviously, $y_\epsilon$ and $y$ satisfy
	\begin{equation*}
		\label{eq:subtracting-rewritten}
		A (y_\epsilon - y) + a_\epsilon (y_\epsilon - y)   = f(y_\epsilon) - f_\epsilon(y_\epsilon) \, \text{in } \Omega, \quad
		y_\epsilon - y  =0 \, \text{on } \partial\Omega.
	\end{equation*}
	Applying the Stampacchia theorem \cite[Thm.~12.4]{Chipot2009} gives
	\begin{equation} \nonumber
		\norm{y_\epsilon - y}_{L^\infty(\Omega)} \leq C \norm{f(y_\epsilon) - f_\epsilon(y_\epsilon)}_{W^{-1,p}(\Omega)}  \leq C \norm{f(y_\epsilon) - f_\epsilon(y_\epsilon)}_{L^\infty(\Omega)}
	\end{equation} 
	for some $p>N$ and for some positive constant $C$ independent of $\epsilon$. The combination of this with \eqref{eq:f_regu} and the fact that $\norm{y_\epsilon}_{C(\overline\Omega)}, \norm{y}_{C(\overline\Omega)} \leq M$ yields \eqref{eq:Linfty-esti}.

		It remains to prove \ref{item:w2s-converge-main}. For that purpose, we assume that $u_\epsilon \rightharpoonup u$ in $L^\barr(\Omega)$ and set $y := S(u)$ and $y_\epsilon := S_\epsilon(u_\epsilon)$. Subtracting the equations for $y_\epsilon$ and $y$ yields
	\begin{equation}
		\label{eq:subtracting-y-reg}
		\left\{
		\begin{aligned}
			A (y_\epsilon - y) + (f(y_\epsilon) - f(y))  & = u_\epsilon - u + f(y_\epsilon) - f_\epsilon(y_\epsilon) && \text{in } \Omega, \\
			y_\epsilon - y &=0 && \text{on } \partial\Omega.
		\end{aligned}
		\right.
	\end{equation}
	Since $\barr > \frac{N}{2}$, there exists a number $p >N$ such that $L^\barr(\Omega) \Subset W^{-1,p}(\Omega)$.
	Applying the Stampacchia theorem to \eqref{eq:subtracting-y-reg} then yields
	\begin{equation} \nonumber
	\norm{y_\epsilon - y}_{L^\infty(\Omega)} \leq C \left[ \norm{u_\epsilon - u}_{W^{-1,p}(\Omega)} + \norm{f(y_\epsilon) - f_\epsilon(y_\epsilon)}_{W^{-1,p}(\Omega)}  \right],
	\end{equation} 
	which together with \eqref{eq:f_regu} gives $y_\epsilon \to y$ in $L^\infty(\Omega)$.
	On the other hand, testing the above equation by $y_\epsilon - y$, applying the monotonicity of $f$, and using the strong ellipticity condition in \cref{ass:ellipticity}, we have
	\begin{align*}
	\Lambda \norm{\nabla (y_\epsilon - y)}_{L^2(\Omega)}^2 & \leq C_\barr\norm{u_\epsilon - u}_{L^\barr(\Omega)} \norm{y_\epsilon -y}_{L^\infty(\Omega)} + \int_\Omega [f(y_\epsilon) - f_\epsilon(y_\epsilon)](y_\epsilon - y) \dx.
	\end{align*}
	From this, \eqref{eq:f_regu}, and the limit  $y_\epsilon \to y$ in $L^\infty(\Omega)$, we have $\norm{\nabla (y_\epsilon - y)}_{L^2(\Omega)} \to 0$. 
	Assertion \ref{item:w2s-converge-main} is then proven.
\end{proof}
\medskip 

	In the remainder of this section, 
let $\bar u := \bar u_\kappa$ be an arbitrary but fixed minimizer of \eqref{eq:P}. In order to derive the first-order optimality conditions at $\bar u$, we consider the regularized optimal control problem
\begin{equation} \label{eq:P-reg}
	\tag{P$_{\kappa,\epsilon}$} 
	\min_{u\in U_{ad}}  H_\epsilon(u) := J_{\kappa,\epsilon}(u)  + \frac{1}{2} \norm{u - \bar u}_{L^2(\Omega)}^2,
\end{equation}
where
\begin{equation} \nonumber
	J_{\kappa,\epsilon} (u):=  F_\epsilon(u) + \kappa j(u) \quad \text{with} \quad F_\epsilon (u) := \int_\Omega L(x, S_\epsilon(u)(x)) \dx + \frac{\nu}{2} \norm{u}_{L^2(\Omega)}^2.
\end{equation}
As a result of the choice of $\barr$ in \eqref{eq:barr-choice}, the functional $ L^\barr(\Omega) \ni u \mapsto \norm{u}_{L^2(\Omega)}^2 \in \Rbb$ is continuously differentiable.
Thanks to assertion \ref{item:S-reg-deri} in \cref{thm:S-reg} and the chain rule,  $F_\epsilon$ is thus of class $C^1$ as a functional on $L^\barr(\Omega)$. By combining this with the convexity and Lipschitz continuity of the functional $j$, we can apply the classical abstract results presented in \cite{BonnansShapiro2000,Clarke2} to deduce the necessary optimality conditions for \eqref{eq:P-reg}.
\begin{proposition}[{cf. \cite[Thm.~3.1]{CasasHerzogWachsmuth2012}}] \label{prop:OS-reg}
	The regularized problem \eqref{eq:P-reg} admits at least one local minimizer $\bar u_\epsilon \in U_{ad}$. Moreover, there exist $\bar p_\epsilon \in H^1_0(\Omega) \cap C(\overline\Omega)$ and $\bar \lambda_\epsilon \in \partial j(\bar u_\epsilon)$ such that
	\begin{subequations}
		\label{eq:1st-OS-reg}
		\begin{align}
			& A^* \bar p_\epsilon +  f_\epsilon'(\bar y_\epsilon) \bar p_\epsilon  = \frac{\partial L}{\partial y}(x,\bar y_\epsilon) \, \text{in } \Omega, \quad
			\bar p_\epsilon =0 \, \text{on } \partial\Omega,
			\label{eq:adjoint-OS-reg} \\
			&\int_\Omega \left(\bar p_\epsilon + \nu \bar u_\epsilon + (\bar u_\epsilon - \bar u) + \kappa \bar \lambda_\epsilon\right)\left( u - \bar u_\epsilon \right) \dx \geq 0 \quad \text{for all } u \in {U}_{ad}  \label{eq:normal-OS-reg}
		\end{align}
	\end{subequations}
	with $\bar y_\epsilon := S_\epsilon(\bar u_\epsilon)$. 
\end{proposition}
We now have everything in hand to prove the first-order optimality system, which is the main result in the section.
The presence of Clarke’s subdifferential $\partial_C$ in the following conditions justifies the term \emph{C-stationarity}
conditions.
\begin{theorem}[first-order optimality system after passing to the limit; C-stationarity] \label{thm:1st-OS}
	Let $\bar u$ be a local minimizer of \eqref{eq:P} and let $\bar y := S(\bar u)$. 
	Then there exist an adjoint state $\bar p \in H^1_0(\Omega) \cap C(\overline\Omega)$ and multipliers $\bar \lambda \in \partial j(\bar u)$ and $\chi \in L^\infty(\Omega)$ satisfying
	\begin{subequations}
		\label{eq:1st-OS}
		\begin{align}
			&
			A^* \bar p +  \chi \bar p  = \frac{\partial L}{\partial y}(x,\bar y) \, \text{in } \Omega, \quad
			\bar p  =0 \, \text{on } \partial\Omega,
			\label{eq:adjoint-OS} \\
			& \chi(x) \in \partial_C f(\bar y(x)) \quad \text{for a.a. } x \in \Omega, \label{eq:Clarke-multiplier} \\
			&\int_\Omega \left(\bar p + \nu \bar u + \kappa \bar \lambda \right)\left( u - \bar u \right) \dx \geq 0 \quad \text{for all } u \in {U}_{ad}. \label{eq:normal-OS}
		\end{align}
	\end{subequations}
	Moreover, if $\lambda^N(\{ \bar y \in E_f\}) =0$, then $\chi(x) = \1_{\{ \bar y \notin E_f \}}(x) f'(\bar y(x))$ for a.a. $x \in \Omega$.
\end{theorem}   
\begin{proof}
	The last assertion is a direct consequence of \eqref{eq:Clarke-multiplier}. We now prove the optimality system \eqref{eq:1st-OS}. To this end, we employ some arguments similar to those used to investigate a non-smooth semilinear optimal control problem without the $L^1$-term \cite[Thm.~4.4]{Constantin2018}.
	Since $\bar u$ is a local minimizer of \eqref{eq:P}, there exists a constant $\rho>0$ such that
	\begin{equation} \nonumber
		J_\kappa(\bar u) \leq J_\kappa(u) = F(u) + \kappa j(u) \quad \text{for all } u \in U_{ad} \, \text{with } \norm{u - \bar u}_{L^2(\Omega)} \leq \rho.
	\end{equation} 
	Consider the following auxiliary optimal control problem of \eqref{eq:P-reg}:
	\begin{equation}
		\label{eq:P-reg-auxi}
		\tag{P$_{\kappa,\epsilon,\rho}$} 
		\min\{ H_\epsilon(u): u \in U_{ad}, \norm{ u - \bar u}_{L^2(\Omega)} \leq \rho \}.
	\end{equation}
	Note that $H_\epsilon$ depends also on the parameter $\kappa$.
	Obviously, there exists at least one global minimizer $\bar u_\epsilon$ of \eqref{eq:P-reg-auxi}. Moreover, as a result of assertion \ref{it:Lipschitz-app} in \cref{prop:control-to-state} and of assertion \ref{item:S-regularization} in \cref{thm:S-reg} there is a constant $M>0$ such that
	\begin{equation}
		\label{eq:y-bound-uniform}
		\norm{S(u)}_{C(\overline\Omega)}, \norm{S_\epsilon(u)}_{C(\overline\Omega)} \leq M \quad \text{for all } u \in U_{ad}.
	\end{equation}
	Then for any $u \in U_{ad}$, there holds
	\begin{align*}
		|J_{\kappa,\epsilon}(u) - J_\kappa(u)|&  = |F_\epsilon(u) - F(u)| 
		\leq \int_\Omega | L(x, S_\epsilon(u)) - L(x, S(u)) | \dx \\
		& \leq  \int_\Omega \int_0^1 |\frac{\partial L}{\partial y}(x, S(u) + \theta (S_\epsilon(u) - S(u)))(S_\epsilon(u) - S(u))| d\theta \dx \\
		& \leq \norm{\phi_M}_{L^2(\Omega)} \norm{S_\epsilon(u) - S(u)}_{L^2(\Omega)}
	\end{align*}
	with $\phi_M$ given in \cref{ass:integrand}. Combining this with \ref{item:S-regularization} in \cref{thm:S-reg} yields for  all $u \in U_{ad}$ and for some constant $C_M >0$ that
	$|J_{\kappa,\epsilon}(u) - J_\kappa(u)| \leq C_M \epsilon$. 
	In particular, there holds
	\begin{equation} \nonumber
		H_\epsilon (\bar u) = J_{\kappa,\epsilon} (\bar u) \leq C_M \epsilon + J_\kappa(\bar u).
	\end{equation} 
	Furthermore, for any $u \in U_{ad}$, one has
	\begin{align*}
		H_\epsilon(u) & = J_{\kappa,\epsilon}(u) + \frac{1}{2}\norm{u - \bar u}_{L^2(\Omega)}^2  \geq J_\kappa(u) + \frac{1}{2}\norm{u - \bar u}_{L^2(\Omega)}^2 - C_M \epsilon.
	\end{align*}
	We then deduce from the optimality of $\bar u$ that 
	\begin{align*}
		H_\epsilon (\bar u) & \leq C_M \epsilon + J_\kappa(\bar u)  \leq C_M \epsilon + J_\kappa(u) \leq H_\epsilon(u) - \frac{1}{2}\norm{u - \bar u}_{L^2(\Omega)}^2 + 2C_M \epsilon
	\end{align*}
	for any $u \in U_{ad}$ with $\norm{ u - \bar u}_{L^2(\Omega)} \leq \rho$. This implies that
	\begin{equation} \nonumber
		H_\epsilon (\bar u) < H_\epsilon(u) \quad \text{for all } u \in U_{ad} \, \text{with } \sqrt{4C_M \epsilon} < \norm{ u - \bar u}_{L^2(\Omega)} \leq \rho.
	\end{equation} 
	We can conclude from the optimality of $\bar u_\epsilon$ that
	\begin{equation}
		\label{eq:u-reg-esti}
		\norm{ \bar u_\epsilon - \bar u}_{L^2(\Omega)} \leq \sqrt{4C_M \epsilon}.
	\end{equation}
	We now prove that $\bar u_\epsilon$ is a local minimizer of \eqref{eq:P-reg} for sufficiently small $\epsilon >0$. To that end, let $u \in U_{ad}$ be arbitrary with $\| u - \bar u_\epsilon \|_{L^2(\Omega)} < \frac{\rho}{2}$. For $\epsilon$ small enough, we have
	\begin{equation*}
		\| u - \bar u \|_{L^2(\Omega)} \leq \| u - \bar u_\epsilon \|_{L^2(\Omega)} + \| \bar u - \bar u_\epsilon \|_{L^2(\Omega)} < \frac{\rho}{2} + \frac{\rho}{2} =\rho,
	\end{equation*}
	which gives that $u$ is a feasible point of the problem \eqref{eq:P-reg-auxi}. The global optimality of $\bar u_\epsilon$ for \eqref{eq:P-reg-auxi} therefore implies that
		$H_\epsilon(\bar u_\epsilon) \leq H_\epsilon(u)$. 
	This show that $\bar u_\epsilon$ is a local minimizer of the problem \eqref{eq:P-reg}.
	On the other hand, by extracting a subsequence, denoted in the same way, we derive from \eqref{eq:u-reg-esti} that $\bar u_\epsilon \to \bar u$ a.a. in $\Omega$. By combining this with the boundedness in $L^\infty(\Omega)$ of $\{\bar u_\epsilon\}$, we deduce from Lebesgue's dominated convergence theorem that
	\begin{equation}
		\label{eq:u-epsilon-limit-barr}
		\norm{ \bar u_\epsilon - \bar u}_{L^\barr(\Omega)} \to 0.
	\end{equation}

	It remains to prove the existence of the desired adjoint state and multipliers. For that purpose, by virtue of  \cref{prop:OS-reg}, there exist $\bar p_\epsilon \in H^1_0(\Omega) \cap C(\overline\Omega)$, $\bar \lambda_\epsilon \in \partial j(\bar u_\epsilon)$ that satisfy \eqref{eq:1st-OS-reg}. Since $|\bar \lambda_\epsilon| \leq 1$ a.a. in $\Omega$, we can take a subsequence if necessary and deduce that
	\begin{equation}
		\label{eq:lambda-converge}
		\bar \lambda_\epsilon \rightharpoonup^* \bar \lambda \quad \text{in} \quad L^\infty(\Omega)
	\end{equation}
	for some $\bar \lambda \in L^\infty(\Omega)$. From 
		the definition of the subdifferential 
	of a convex functional, there holds
	\begin{equation} \nonumber
		\int_\Omega \bar\lambda_\epsilon(x)(u(x) - \bar u_\epsilon(x)) \dx \leq j(u) - j(\bar u_\epsilon) \quad \text{for all } u \in L^1(\Omega).
	\end{equation} 
	Letting $\epsilon \to 0^+$ and using \eqref{eq:u-epsilon-limit-barr} and \eqref{eq:lambda-converge}, one has
	\begin{equation} \nonumber
		\int_\Omega \bar\lambda(x)(u(x) - \bar u(x)) \dx \leq j(u) - j(\bar u) \quad \text{for all } u \in L^1(\Omega),
	\end{equation} 
	which yields $\bar \lambda \in \partial j(\bar u)$. Besides, combining \eqref{eq:u-epsilon-limit-barr} with \ref{item:w2s-converge-main} in \cref{thm:S-reg} yields 
	\begin{equation}
		\label{eq:y-convergence}
		\bar y_\epsilon \to \bar y \quad \text{in } H^1_0(\Omega) \cap C(\overline\Omega).
	\end{equation}
	From \eqref{eq:y-bound-uniform} and \eqref{eq:f_regu_Lip}, we derive that
	$|f_\epsilon'(\bar y_\epsilon(x))| \leq C$ for a.a. $x \in \Omega$ and for some constant $C>0$ independent of $\epsilon$. We can thus extract a subsequence, denoted in the same way, such that
	\begin{equation} \label{eq:Clarke_conv}
		f_\epsilon'(\bar y_\epsilon) \rightharpoonup^* \chi \quad \text{in } L^\infty(\Omega) \quad \text{as } \epsilon \to 0^+
	\end{equation}
	for some $\chi\in L^\infty(\Omega)$. 
	Combining this with the limit \eqref{eq:y-convergence} and the definition of $f_\epsilon$, we derive from \cite[Chap.~I, Thm.~3.14]{Tiba1990} that $\chi(x) \in \partial_C f(\bar y(x))$ for a.a. $x \in \Omega$. Then \eqref{eq:Clarke-multiplier} follows. 
	
	We now test \eqref{eq:adjoint-OS-reg} by $\bar p_\epsilon$ and use the non-negativity of $f_\epsilon'$ and the ellipticity condition in \cref{ass:ellipticity} to have
	\begin{equation} \nonumber
		\Lambda \norm{\nabla \bar p_\epsilon}_{L^2(\Omega)}^2 \leq \int_\Omega \frac{\partial L}{\partial y}(x, \bar y_\epsilon) \bar p_\epsilon \dx \leq  \norm{\bar \phi_M }_{L^2(\Omega)}\norm{\bar p_\epsilon}_{L^2(\Omega)} 
	\end{equation} 
	with $\phi_M$ defined as in \cref{ass:integrand}. This and the  Poincar\'{e} inequality show the boundedness of $\{\bar p_\epsilon \}$ in $H^1_0(\Omega)$. We can thus extract a subsequence, denoted in the same way, such that
	\begin{equation}
		\label{eq:p-convergence}
		\bar p_\epsilon \rightharpoonup \bar p \, \text{in } H^1_0(\Omega) \quad \text{and} \quad \bar p_\epsilon \to \bar p \, \text{in } L^2(\Omega)
	\end{equation}
	for some $\bar p \in H^1_0(\Omega)$. Letting $\epsilon \to 0^+$ in \eqref{eq:adjoint-OS-reg} and \eqref{eq:normal-OS-reg} and exploiting \eqref{eq:u-epsilon-limit-barr}, \eqref{eq:lambda-converge}, \eqref{eq:y-convergence},  \eqref{eq:Clarke_conv}, and \eqref{eq:p-convergence}, we obtain \eqref{eq:adjoint-OS} and \eqref{eq:normal-OS}, respectively. 
\end{proof}
\begin{remark}
	\label{rem:normal-OS}
	It is noted that \eqref{eq:normal-OS} is equivalent to
	\begin{equation}
		\label{eq:norm-OS-equi}
		\int_\Omega \left(\bar p + \nu \bar u + \kappa \bar \lambda \right) v \dx \geq 0 \quad \text{for all } v \in \textrm{cl}_{L^\barr(\Omega)}(\cone(U_{ad} - \bar u)),
	\end{equation}
	where $\cone(U_{ad} - \bar u)$ denotes the cone generated via $U_{ad} - \bar u$ and $\textrm{cl}_{L^\barr(\Omega)}(\cone(U_{ad} - \bar u))$ stands for the closure in $L^\barr(\Omega)$ of $\cone(U_{ad} - \bar u)$. 	
	Recall from the classical results on variational analysis that
	\begin{equation}
		\label{eq:tangent-cone}
		\textrm{cl}_{L^\barr(\Omega)}(\cone(U_{ad} - \bar u)) = \left\{ v \in L^\barr(\Omega) \mid v \geq 0 \, \text{a.a. in } \{\bar u = \alpha\} \, \text{and } v \leq 0 \, \text{a.a. in } \{\bar u = \beta\} \right\};
	\end{equation}
	see, e.g. \cite[Prop.~2.55]{BonnansShapiro2000} and \cite[Lem.~4.11~(i)]{BayenBonnansSilva2014}.
\end{remark}
The next relations are immediate consequences of \eqref{eq:normal-OS} and the fact that $\alpha < 0 < \beta$. 
\begin{corollary}
	\label{cor:projections}
	Let $\bar u$, $\bar y$, $\bar p$, $\bar \lambda$, and $\chi$ be defined as in the previous theorem. Then, for a.a. $x \in \Omega$, the following relations hold:
	\begin{subequations}
		\label{eq:1st-projection}
		\begin{align}
		& \bar u(x) = \proj_{[\alpha,\beta]}\left(- \frac{1}{\nu}(\bar p(x) + \kappa \bar \lambda(x)) \right), \label{eq:projection-u} \\
		& \bar u(x) = 0 \quad \Leftrightarrow \quad |\bar p(x)| \leq \kappa, \label{eq:sparsity}\\
		\text{and} \quad & \lambda(x) = \proj_{[-1,+1]}\left(-\frac{1}{\kappa} \bar p(x)\right) \quad \text{whenever} \quad \kappa >0. \label{eq:projection-lambda}
		\end{align}
	\end{subequations}
	Moreover, if $\lambda^N(\{ \bar y \in E_f\}) =0$, then the adjoint state $\bar p$ and the multiplier $\bar \lambda$ are uniquely determined for any local minimizer $\bar u$.
\end{corollary}
\begin{proof}
	The argument for \eqref{eq:1st-projection} is standard; see, e.g. \cite[Cor.~3.2]{CasasHerzogWachsmuth2012}. Assume now that $\lambda^N(\{ \bar y \in E_f\}) =0$. Thanks to \cref{thm:1st-OS}, $\chi= \1_{\{ \bar y \notin E_f\}} f'(\bar y)$ a.a. in $\Omega$ and $\bar p$ is thus unique via \eqref{eq:adjoint-OS}. Therefore, $\bar \lambda$ is uniquely determined by the representation formula \eqref{eq:projection-lambda}.
\end{proof}

\begin{remark}
	\label{rem:sparsity}
	As observed in \cite{Stadler2009}, the relation \eqref{eq:sparsity} leads to the sparsity of the local optimal controls: 
	With increasing $\kappa$, 
	the support of the local optimal controls shrinks to have the measure zero for $\kappa$ large enough. 
\end{remark}

\section{Second-order optimality conditions} \label{sec:2ndOS}
One of our main aims of this paper is to derive second-order necessary and sufficient optimality conditions 
in explicit forms, 
	which could be of great interest for showing error estimates for the numerical approximation
of \eqref{eq:P}; see, e.g. \cite{ClasonNhuRosch2022,ClasonNhuRosch2023_part2}. 
For this purpose, we first introduce a non-smooth \emph{curvature functional} that characterizes the  second-order generalized derivative of the smooth ingredient $F$ of objective functional $J$ in critical directions. 
We then establish the second-order optimality conditions containing the curvature functional. Finally, after explicitly computing this curvature, the explicit second-order conditions are consequently derived and include the integrals over $(N-1)$-dimensional manifolds.

\subsection{Non-smooth curvature functional}\label{sec:2nd-OC:curvature}
Before establishing the non-smooth curvature functional of $F$ we need some additional notation.
For any $t>0$ and $u, h \in L^\barr(\Omega)$
we now define the following function on $\Omega$
\begin{multline*} 
	\zeta(u;t,h)  = - \sum_{i=1}^\barK \sigma_i (S(u + th) - \tau_i) \left[\1_{ \{S(u) \in (\tau_i, \tau_{i+1}), S(u + th)  \in (\tau_{i}-\epsilon_0, \tau_i) \}} \right. \\
	\left. - \1_{ \{S(u) \in (\tau_{i-1}, \tau_{i}),S(u + th)  \in (\tau_{i}, \tau_i+ \epsilon_0) \}} \right],
\end{multline*}
where $\epsilon_0$ is defined as in \eqref{eq:epsilon-zero}.
We then determine, for $\{t_n\} \in c_0^+$, $u, h\in L^\barr(\Omega)$, $p \in L^1(\Omega)$, and $\{h_n\} \subset L^\barr(\Omega)$, the extended numbers
\begin{equation}
	\label{eq:key-term-sn}
	\underline{Q}( u, p;\{t_n\}, \{h_n\}) := \liminf_{n \to \infty} \frac{1}{t_n^2} \int_\Omega  p  \zeta ( u;t_n, h_n) \dx
\end{equation}
and
\begin{equation*}
	\tilde{Q}( u, p; h) := \inf \left\{ \underline{Q}( u, p;\{t_n\}, \{h\}) \mid \{t_n\} \in c_0^+ \right\}.
\end{equation*}
Here $\{h\}$ stands for the constant sequence $h_n \equiv h$ for all $n\geq 1$.

\begin{definition}
	\label{def:curvature-func-app}
	The non-smooth curvature functional $Q$ of $F$ at $u$ for $p \in L^1(\Omega)$ is defined as an extended mapping from $L^\barr(\Omega)$ to $[-\infty, \infty]$ and given by
	\begin{multline*}
		Q( u,  p; h) :=  \int_\Omega[ \frac{\partial^2 L}{\partial y^2}(x, S(u)) \left(S'( u; h)\right)^2 + \nu h^2 ] \dx \\
		\begin{aligned}[t]
			- \int_\Omega \1_{\{ S(u) \notin E_f \}}  p f''_{yy}( S(u))\left(S'( u; h)\right)^2 \dx +  2\tilde{Q}( u, p; h), \quad h \in L^\barr(\Omega).
		\end{aligned} 
	\end{multline*}
\end{definition}

\begin{remark}
	\label{rem:max-pde}
	When $\barr =2$ and the state equation is defined by \eqref{eq:max-pde}, i.e., $A = -\Delta$, $f(t) = \max\{0,t\}$, then $K=1$, $\tau_1 = 0$, $f_0(t) = 0$,  $f_1(t) =t$,  $\sigma_1 = -1$, and thus 
	\begin{equation} \nonumber
		Q( u,  p; h) =  \int_\Omega [ \frac{\partial^2 L}{\partial y^2}(x, S(u)) \left(S'( u; h)\right)^2 + \nu h^2 ] \dx  +  2\tilde{Q}( u, p; h).
	\end{equation} 
	This coincides with the one in \cite[Def.~5.1]{Nhu2021Optimization}. 
	
	Moreover, if $f$ is of class $C^2$, then $E_f = \emptyset$ and thus $Q( u,  p; h) = F''( u)h^2$ for $p = G_{\chi} (\frac{\partial L}{\partial y}(\cdot, S(u)))$ with $\chi = f'(S(u))$, see, e.g. \cite[Thm.~2.3]{Casas2008}. In a more general case,  \cref{prop:Q-weak2subder} shows that $Q$ is identical to the strong second subderivative of $F$ under the structural assumption \eqref{eq:structure-non-diff} below.
\end{remark}

	In the remainder of this section, 
let $\bar u \in U_{ad}$ be an admissible control and let $\bar y:= S(\bar u)$. We need the following structural condition on $\bar y$ to provide some required properties of $\tilde{Q}$. 
\begin{assumption2} 
	\item \label{ass:structure-non-diff}
	Assume that the function $\bar y \in C(\overline\Omega)$ has the property that 
	\begin{equation}
		\label{eq:structure-non-diff}
		\tag{SA}
		\sum_{i =1}^K\lambda^N\left( \{ |\bar y - \tau_i| < \epsilon \} \right) \leq c_s \epsilon
	\end{equation}
	for all $\epsilon \in (0,1)$ {and for some constant $c_s >0$.}
\end{assumption2}
The same condition as \eqref{eq:structure-non-diff} but imposed on an adjoint state was employed in \cite{WachsmuthWachsmuth2011} to derive error
	estimates w.r.t. the regularization parameter 
for an elliptic optimal control problem. Such a condition was also employed in  \cite{Wachsmuth2013}  to deal with the parameter choice rule for the Tikhonov regularization parameter depending on a posteriori computable quantities. 
In \cite{DeckelnichHinze2012}, a more general version of \eqref{eq:structure-non-diff}  was used to show the a priori error estimates for the approximation of elliptic control problems. 

The assumption \eqref{eq:structure-non-diff} requires that the  set $\{\bar y \in E_f \}$ has measure zero, which implies the G\^{a}teaux-differentiability at $\bar u$ of $S$ according to \cref{prop:G-diff-control2state}.
Thanks to \cref{cor:G-diff},  the functional $F$ is  G\^{a}teaux-differentiable in $\bar u$. Moreover, the derivative in $\bar u$ of $F$ is defined as
\begin{equation*}
	F'(\bar u)h = \int_\Omega (\bar p + \nu \bar u)h \dx 
\end{equation*}
for all $h \in L^\barr(\Omega)$, where  
\begin{equation}
	\label{eq:adjoint-state-SA}
	\bar p = G^*_{ \chi}(\frac{\partial L}{\partial y}(\cdot, \bar y)) \in H^1_0(\Omega) \cap C(\overline\Omega)
\end{equation}
with
\begin{equation}
	\label{eq:chi-relation-SA}
	\chi(x) = \1_{\{ \bar y \notin E_f \}} (x)f'(\bar y(x)) \quad \text{for a.a. } x \in \Omega.
\end{equation}
It is noted that under the condition $\lambda^N(\{ \bar y \in E_f \}) =0$, the relation \eqref{eq:Clarke-multiplier} is then identical to \eqref{eq:chi-relation-SA} and thus the adjoint state $\bar p$ determined by \eqref{eq:adjoint-state-SA} satisfies \eqref{eq:adjoint-OS}.

The following result shows the definiteness of $\tilde{Q}$ at $\bar u$ under \eqref{eq:structure-non-diff}.
\begin{lemma}
	\label{lem:Q-defined}
	Under the structural assumption \eqref{eq:structure-non-diff}, there holds    
	\begin{equation*}
		| \tilde{Q}(\bar u,\bar p;h) | \leq c_s  \sigma_{\max} \norm{\bar p}_{L^\infty(\Omega)} \norm{S'(\bar u)h}_{L^\infty(\Omega)}^2
	\end{equation*}    
	for any $h \in L^\barr(\Omega)$ with $\sigma_{\max} := \max\{|\sigma_i|: 1 \leq i \leq K \}$.
\end{lemma}
\begin{proof}
	By the definition of $\tilde{Q}$, it suffices to prove for any $\{t_n\} \in c_0^+$ that
	\begin{equation}
		\label{eq:Q-bound}
		\left| \underline Q(\bar u,\bar p;\{t_n\},\{h\}) \right| \leq c_s  \sigma_{\max} \norm{\bar p}_{L^\infty(\Omega)} \norm{S'(\bar u)h}_{L^\infty(\Omega)}^2.
	\end{equation}
	To this end, setting $y_n := S(\bar u + t_n h)$ yields $y_n \to \bar y$ in $H^1_0(\Omega) \cap C(\overline\Omega)$. 
	Define the measurable functions
	\begin{equation} \label{eq:Tn-i}
		T_n^{i} :=  (y_n-\tau_i  )( \1_{\{\bar y \in (\tau_i, \tau_{i+1}), y_n \in (\tau_i-\epsilon_0, \tau_i) \}}  - \1_{\{ y_n \in (\tau_{i}, \tau_i+\epsilon_0) , \bar y \in (\tau_{i-1},\tau_i) \}} )
	\end{equation}
	with $1 \leq i \leq K, n \geq 1.$
	We then have from the definition of $\zeta$ that
	\begin{equation}
		\label{eq:zeta-func-expression}
		\zeta(\bar u; t_n, h) = -\sum_{i =1}^K \sigma_i T_n^i.
	\end{equation}
	Moreover, from the definition of $T_n^i$, there holds
	\begin{align*}
		0 \geq {T_n^i} \geq   (y_n-\bar y) ( \1_{\{\bar y \in (\tau_i, \tau_{i+1}), y_n \in (\tau_i-\epsilon_0, \tau_i) \}}  - \1_{\{ y_n \in (\tau_{i}, \tau_i+\epsilon_0) , \bar y \in (\tau_{i-1},\tau_i) \}} )
	\end{align*}
	a.a. in $\Omega$. By setting $\epsilon_n := \norm{y_n - \bar y}_{C(\overline\Omega)}$, a detailed computation yields
	\begin{equation}
		\label{eq:inclusion-key1}
		\{\bar y \in (\tau_i, \tau_{i+1}), y_n \in (\tau_i-\epsilon_0, \tau_i) \} \cup \{ y_n \in (\tau_{i}, \tau_i+\epsilon_0) , \bar y \in (\tau_{i-1},\tau_i) \} \subset \{  |\bar y - \tau_i| < \epsilon_n \}
	\end{equation}
	and thus
	\begin{equation} \label{eq:T-n-esti-almost}
		0 \geq T_n^i \geq - |y_n - \bar y| \1_{ \{ |\bar y - \tau_i| < \epsilon_n \} }. 
	\end{equation}
	Noting that $\epsilon_n \to 0^+$ and applying the structural assumption \eqref{eq:structure-non-diff} then yield
	\begin{equation*}
		\sum_{i =1}^K \norm{T_n^i}_{L^1(\Omega)} \leq c_s \epsilon_n^2 = c_s \norm{y_n - \bar y}_{C(\overline\Omega)}^2 \quad \text{for $n$ large enough}.
	\end{equation*}
	This implies that
	\begin{equation} \label{eq:zeta-esti}
		|\int_\Omega \bar p \zeta(\bar u;t_n, h) \dx | = | \int_\Omega \sum_{i =1}^K \sigma_i T_n^i \bar p \dx| \leq c_s\norm{\bar p}_{L^\infty(\Omega)} \epsilon_n^2 \sigma_{\max} 
	\end{equation}
	for $n$ sufficiently large. 
	Besides, it follows from \cref{prop:G-diff-control2state} that $S$ is G\^{a}teaux-differentiable at $\bar u$ and thus
	$\frac{\epsilon_n}{t_n} \to \norm{S'(\bar u)h}_{L^\infty(\Omega)}$.
	We then derive \eqref{eq:Q-bound} from \eqref{eq:zeta-esti} and the definition \eqref{eq:key-term-sn} of $\underline Q$. 
\end{proof}

The positive homogeneity of degree $2$ in $h$ of $Q$ is stated in the following lemma. Its proof is similar to that in \cite[Lem.~5.2]{Nhu2021Optimization} as well as  \cite[Lem.~5.1]{ClasonNhuRosch2020} 
	(see, also \cite[Lem.~3.2 (i)]{ChristofWachsmuth2018})
and thus omitted. 

\begin{lemma} \label{lem:homogeneity}
	For any $h \in L^\barr(\Omega)$ and any $ t \geq 0$, there hold 
	\begin{equation} \nonumber
		\tilde{Q}(\bar u, \bar p;th) = t^2 \tilde{Q}(\bar u, \bar p;h) \, \text{and} \, Q(\bar u, \bar p;th)  = t^2 Q(\bar u, \bar p;h).
	\end{equation} 
\end{lemma}

The following shows that $\underline{Q}(\bar u, \bar p; \{t_n\}, \cdot)$ is invariant for two sequences possessing the same weak limit. This ensures the weak lower semicontinuity of $\tilde{Q}(\bar u, \bar p; \{t_n\}, \cdot)$ as well as 
${Q}(\bar u, \bar p; \cdot)$. 
\begin{lemma}
	\label{lem:invariant}
	Assume that \eqref{eq:structure-non-diff} is fulfilled. Then, for any sequences $\{t_n\} \in c_0^+$ and $\{h_n\}, \{v_n\} \subset L^\barr(\Omega)$ such that $h_n \rightharpoonup h$ and $v_n \rightharpoonup h$ in $L^\barr(\Omega)$ for some $h \in L^\barr(\Omega)$, there holds
	\begin{equation} \label{eq:weak-invariant}
		\lim_{n \to \infty} \frac{1}{t_n^2} \int_\Omega \bar p \left[ \zeta(\bar u; t_n, h_n) -\zeta(\bar u; t_n, v_n)   \right] \dx = 0
	\end{equation}
	and, in particular, 
	$ 
		\underline{Q}(\bar u, \bar p; \{t_n\},\{h_n\}) = \underline{Q}(\bar u, \bar p; \{t_n\},\{h\}) \geq \tilde{Q}(\bar u, \bar p; h).
	$ 
\end{lemma}
\begin{proof}
	It is sufficient to prove \eqref{eq:weak-invariant}. For this purpose, we set $y_n := S(\bar u + t_n h_n)$, $z_n := S(\bar u + t_n v_n)$ and $\epsilon_n := \norm{y_n - \bar y}_{C(\overline\Omega)}$, $\kappa_n := \norm{z_n - \bar y}_{C(\overline\Omega)}$, $\rho_n :=\norm{y_n-z_n}_{C(\overline\Omega)}$. Obviously, we deduce from \cref{prop:control-to-state,prop:G-diff-control2state} that 
	\begin{equation}
		\label{eq:limits-invar}
		\epsilon_n, \kappa_n \to 0^+, \quad \frac{\epsilon_n}{t_n}, \frac{\kappa_n}{t_n} \to \norm{S'(\bar u)h}_{L^\infty(\Omega)}, \quad \text{and} \quad \frac{\rho_n}{t_n} \to 0.
	\end{equation}
	Analogous to \eqref{eq:zeta-func-expression}, there holds 
	\begin{equation}
		\label{eq:zeta-func-expression-vn}
		\zeta(\bar u; t_n, v_n) = -\sum_{i =1}^K \sigma_i \tilde T_n^i
	\end{equation}
	with 	
	$ 
		\tilde T_n^{i} :=  (z_n-\tau_i  )( \1_{\{\bar y \in (\tau_i, \tau_{i+1}), z_n \in (\tau_i-\epsilon_0, \tau_i) \}}  - \1_{\{ z_n \in (\tau_{i}, \tau_i+\epsilon_0) , \bar y \in (\tau_{i-1},\tau_i) \}} ) 
	$ 
	for $1 \leq i \leq K, n \geq 1$.
	From \eqref{eq:zeta-func-expression} and \eqref{eq:zeta-func-expression-vn}, we have 
	\begin{equation}
		\label{eq:iden-invariant}
		\bar p \left[\zeta(\bar u; t_n, h_n) -  \zeta(\bar u; t_n, v_n) \right]= -\sum_{i=1}^K \sigma_i \bar p (A_n^i + B_n^i)
	\end{equation} 
	a.a. in $\Omega$ and for all $n \geq 1$, where
	\begin{equation*}
		A_n^i:=  (y_n - z_n)  ( \1_{\{\bar y \in (\tau_i, \tau_{i+1}), y_n \in (\tau_i-\epsilon_0, \tau_i) \}}  - \1_{\{ y_n \in (\tau_{i}, \tau_i+\epsilon_0) , \bar y \in (\tau_{i-1},\tau_i) \}} )
	\end{equation*}
	and
	\begin{multline*}
		B_n^i := (z_n -\tau_i) \left\{ \1_{ \{ \bar y \in (\tau_i, \tau_{i+1}) \} }[ \1_{ \{ y_n \in (\tau_i-\epsilon_0, \tau_i)\} } - \1_{\{  z_n \in (\tau_i-\epsilon_0, \tau_i)\}} ] \right.\\
		\left. 	-  \1_{\{ \bar y \in (\tau_{i-1},\tau_i)\} }[ \1_{\{ y_n \in (\tau_{i}, \tau_i+\epsilon_0)\}} - \1_{\{ z_n \in (\tau_{i}, \tau_i+\epsilon_0)\} } ]\right\}.
	\end{multline*}
	By using \eqref{eq:inclusion-key1},  there holds
	$ 
		|A_n^i| \leq |y_n - z_n| \1_{\{ |\bar y - \tau_i| < \epsilon_n \}},
	$ 
	which, together with  \eqref{eq:structure-non-diff}, yields
	\begin{equation} \label{eq:A-n}
		\sum_{i =1}^K \norm{A_n^i}_{L^1(\Omega)} \leq c_s \norm{y_n  - z_n}_{C(\overline\Omega)} \epsilon_n = c_s \rho_n \epsilon_n
	\end{equation}
	for all $n$ large enough.
	For the estimate of  $B_n^i$, we rewrite $B_n^i$ as follows
	\begin{multline*}
		B_n^i := (z_n -\tau_i)  \1_{\{ \bar y \in (\tau_i, \tau_{i+1}) \} }  \1_{\{  y_n \in (\tau_i-\epsilon_0, \tau_i), z_n \in [\tau_i, \tau_i + \rho_n] \} } \\
		\begin{aligned}
			&-(z_n -\tau_i)  \1_{\{ \bar y \in (\tau_i, \tau_{i+1}) \} }  \1_{ \{ z_n \in (\tau_i-\epsilon_0, \tau_i), y_n \in [\tau_i, \tau_i + \rho_n] \}}   \\
			&- (z_n -\tau_i) \1_{\{ \bar y \in (\tau_{i-1},\tau_i)\} } \1_{\{ y_n \in (\tau_{i}, \tau_i+\epsilon_0), z_n \in [\tau_i-\rho_n, \tau_i]\}} \\
			&+(z_n -\tau_i) \1_{\{ \bar y \in (\tau_{i-1},\tau_i)\} } \1_{\{ z_n \in (\tau_{i}, \tau_i+\epsilon_0), y_n \in [\tau_i -\rho_n, \tau_i] \}} 
		\end{aligned}
	\end{multline*}
	for $n$ large enough. 
	For a.a. $x \in \{ \bar y \in (\tau_i, \tau_{i+1})  \} \cap \{  y_n \in (\tau_i-\epsilon_0, \tau_i), z_n \in [\tau_i, \tau_i + \rho_n] \}$, we have
	\begin{equation*}
		0 \leq z_n(x) - \tau_i \leq |z_n(x) - y_n(x)| \leq \rho_n \quad \text{and} \quad 0 < \bar y(x) -\tau_i \leq \bar y(x) - y_n(x) \leq  \epsilon_n
	\end{equation*}
	and there thus holds
	\begin{equation} \nonumber
		|z_n -\tau_i|  \1_{\{ \bar y \in (\tau_i, \tau_{i+1}) \} } \1_{\{  y_n \in (\tau_i-\epsilon_0, \tau_i), z_n \in [\tau_i, \tau_i + \rho_n] \} } \leq \rho_n \1_{\{0 < \bar y - \tau_i \leq \epsilon_n\}}
	\end{equation}  
	a.a. in $\Omega$. The analogous estimates hold for other terms of $B_n^i$. We then conclude for $n$ large enough that
	$
		|B_n^i| \leq \rho_n [\1_{\{ |\bar y - \tau_i| \leq \epsilon_n \}} +\1_{\{ |\bar y - \tau_i| \leq \kappa_n \}} ],
	$ 
	which, together with \eqref{eq:structure-non-diff}, yields
	\begin{equation} \label{eq:B-n}
		\sum_{i =1}^K \norm{B_n^i}_{L^1(\Omega)} \leq c_s \rho_n (\epsilon_n + \kappa_n).
	\end{equation}
	Combining \eqref{eq:iden-invariant} with \eqref{eq:A-n} and \eqref{eq:B-n}, we arrive at
	\begin{equation*}
		| \int_\Omega \bar p \left[ \zeta(\bar u; t_n, h_n) -\zeta(\bar u; t_n, v_n)   \right] \dx |  \leq c_s  \norm{\bar p}_{L^\infty(\Omega)} \rho_n (2\epsilon_n + \kappa_n) \sigma_{\max}.
	\end{equation*}
	From this and \eqref{eq:limits-invar}, we have \eqref{eq:weak-invariant}.
\end{proof}

We now can employ \cref{lem:invariant} to show the weak lower semi-continuity in $h$ of $\tilde{Q}$. The argument for this is similar to that in \cite[Prop.~5.2]{Nhu2021Optimization} as well as  \cite[Prop.~5.6]{ClasonNhuRosch2020} and is thus skipped.
\begin{proposition}\label{prop:wlsc-key-term}
	If \eqref{eq:structure-non-diff} is fulfilled, then 
	for any $h_n \rightharpoonup h$ in $L^\barr(\Omega)$,
	\begin{equation*}
		\tilde{Q}(\bar u,\bar p;h) \leq \liminf_{n \to \infty} \tilde{Q}(\bar u,\bar p;h_n).
	\end{equation*}
	In other words, $\tilde{Q}(\bar u,\bar p;\cdot)$ and thus $Q(\bar u, \bar p; \cdot)$ are weakly lower semi-continuous as functionals on $L^\barr(\Omega)$.
\end{proposition}

\subsection{Second-order optimality conditions based on $Q$} \label{sec:2ndOS-based-Q}

In order to derive the second-order optimality system, we need the following second-order Taylor-type expansion of $F$ in $\bar u$. 

\begin{lemma}
	\label{lem:2nd-Taylor}
	Assume that $\lambda^N(\{ \bar y \in E_f \}) =0$. Then, for any $u \in L^\barr(\Omega)$, the following second-order Taylor-type expansion is valid
	\begin{multline*}
		F(u) - F(\bar u) = \int_\Omega\int_0^1 (1-s) \frac{\partial^2 L}{\partial y^2}(x,\bar y + s(y_u -\bar y))(y_u - \bar y)^2 \ds \dx \\
		\begin{aligned}[t]
			& + \frac{\nu}{2} \norm{u - \bar u}_{L^2(\Omega)}^2+ \int_\Omega (\bar p + \nu \bar u) (u - \bar u)\dx  + r(y_u),
		\end{aligned}
	\end{multline*}
	with   $y_u:= S(u)$,
	where
	\begin{equation}
		\label{eq:remainder}
		r(y_u) :=  \int_\Omega \bar p \left[ \bar \chi (y_u -\bar y) - f(y_u) + f(\bar y) \right] \dx
	\end{equation}
	with $\bar \chi(x) := \1_{\{ \bar y \notin E_f \}}(x)f'(\bar y(x))$ for a.a. $x \in \Omega$. 
\end{lemma}
\begin{proof}
	Employing a Taylor expansion gives 
	\begin{multline} \label{eq:obj-diff}
		F(u) - F(\bar u)  =  \int_\Omega \left[L(x,y_u) - L(x, \bar y) \right]\dx + \frac{\nu}{2} \int_\Omega (u^2 - \bar u^2)\dx \\
		\begin{aligned}[b]				
			& = \int_\Omega \frac{\partial L}{\partial y}(x, \bar y)(y_u - \bar y) \dx +\nu \int_\Omega  \left( u - \bar u \right)  \bar u\dx  + \frac{\nu}{2} \norm{u - \bar u}_{L^2(\Omega)}^2 \\
			\MoveEqLeft[-1] + \int_\Omega \int_0^1 (1-s) \frac{\partial^2 L}{\partial y^2}(x,\bar y + s(y_u -\bar y))(y_u - \bar y)^2 \ds\dx  \\
			& = \int_\Omega \frac{\partial L}{\partial y}(x, \bar y)(y_u - \bar y) \dx - \int_\Omega \bar p \left( u - \bar u \right)  \dx + \int_\Omega (\bar p + \nu \bar u)(u - \bar u)\dx  \\
			\MoveEqLeft[-1] + \int_\Omega \int_0^1 (1-s)\frac{\partial^2 L}{\partial y^2}(x,\bar y + s(y_u -\bar y))(y_u - \bar y)^2 \ds\dx + \frac{\nu}{2} \norm{u - \bar u}_{L^2(\Omega)}^2.
		\end{aligned}
	\end{multline} 	
	We now test the state equations for $y_u$ and $\bar y$ by $\bar p$ and then subtract the derived results to have
	\begin{equation*}
		\int_\Omega \bar p (u - \bar u) \dx = \int_\Omega \sum_{i,j=1}^N a_{ij}\partial_{i} (y_u -  \bar y) \partial_{j} \bar  p + a_0 (y_u-\bar y) \bar p + \bar p \left[f(y_u) - f(\bar y)\right] \dx.
	\end{equation*}
	Testing the equation \eqref{eq:adjoint-state-SA} for $\bar p$ by $(y_u -\bar y)$ gives
	\begin{multline*}
		\int_\Omega \frac{\partial L}{\partial y}(x, \bar y)(y_u - \bar y) \dx = \int_\Omega \sum_{i,j=1}^N a_{ij}\partial_{i} (y_u -  \bar y) \partial_{j} \bar p \\
		 + a_0 (y_u-\bar y) \bar p +\1_{\{\bar y \notin E_f\}} f'(\bar y) \bar p (y_u - \bar y) \dx.
	\end{multline*}
	Subtracting the two above equations yields
	\begin{equation} \nonumber
		\int_\Omega \frac{\partial L}{\partial y}(x, \bar y)(y_u - \bar y) \dx - \int_\Omega \bar p (u - \bar u) \dx = \int_\Omega \bar p \left[ \bar \chi (y_u -\bar y) - f(y_u) + f(\bar y) \right] \dx.
	\end{equation} 
	This together with \eqref{eq:obj-diff} gives the desired Taylor-type expansion of $F$ at $\bar u$. 
\end{proof}

The crucial analysis is now the limit 
	relevant for 
the remainder term defined in \eqref{eq:remainder}.
\begin{lemma}
	\label{lem:key-limit}
	Assume that  $\lambda^N(\{\bar y \in E_f \}) =0$. Then, for $h_n \rightharpoonup h$ in $L^\barr(\Omega)$ and $\{t_n\} \in c_0^+$, there holds
	\begin{equation*}
		\label{eq:key-limit}
		 \frac{1}{t_n^2}  r(S(\bar u + t_n h_n)) = - \frac{1}{2}\int_\Omega \1_{\{\bar y \notin E_f \}}\bar p f''_{yy}(\bar y)\left(S'(\bar u) h\right)^2\dx  + \frac{1}{t_n^2}\int_\Omega  \bar p \zeta(\bar u; t_n, h_n) \dx+ o(1)
	\end{equation*}
	for $n$ large enough. Consequently, one has
	\begin{equation} \nonumber
		\liminf\limits_{n \to \infty} \frac{1}{t_n^2}  r(S(\bar u + t_n h_n)) = -\frac{1}{2}\int_\Omega \1_{\{\bar y \notin E_f \}}\bar p f''_{yy}(\bar y)\left(S'(\bar u) h\right)^2 \dx + \underline{Q}(\bar u,\bar p; \{t_n\}, \{h_n\}).
	\end{equation} 
\end{lemma}
\begin{proof}
	Setting $y_n := S(\bar u + t_nh_n)$, we deduce from  \eqref{eq:weak-strong-Hadamard-app} and \cref{prop:G-diff-control2state} that
	\begin{equation} \label{eq:G-diff-key}
		\frac{y_n - \bar y}{t_n} \to S'(\bar u)h \quad
		{
		 \text{strongly in} \quad H^1_0(\Omega) \cap C(\overline\Omega).
		}%
	\end{equation}
	Exploiting the fact $\lambda^N(\{\bar y \in E_f \}) =0$, we can write
	\begin{equation*}
		r(y_n) = \sum_{i=0}^\barK \int_\Omega \1_{\{ \bar y \in (\tau_i, \tau_{i+1}) \}} \bar p \left[ f'_i(\bar y) (y_n -\bar y) - f(y_n) + f_i(\bar y) \right] \dx =: \sum_{i=0}^\barK r_{i,n}.
	\end{equation*}
	Since $h_n \rightharpoonup h$ in $L^\barr(\Omega)$, $y_n \to \bar y$ in $C(\overline\Omega)$ and thus $\norm{y_n - \bar y}_{C(\overline\Omega)} <\epsilon_0$ for $n$ large enough. With $n$ sufficiently large, we have for any $0\leq i \leq K$ the decomposition
	\begin{multline*}
		\{ \bar y \in (\tau_i, \tau_{i+1}) \} = \{ \bar y \in (\tau_i, \tau_{i+1}), y_n \in (\tau_i-\epsilon_0, \tau_i] \} \cup \{ \bar y \in (\tau_i, \tau_{i+1}), y_n \in (\tau_i, \tau_{i+1}) \}\\
		\begin{aligned}
			\cup \{ \bar y \in (\tau_i, \tau_{i+1}), y_n \in [\tau_{i+1}, \tau_{i+1}+ \epsilon_0) \} =: \Omega_{i,n}^1 \cup \Omega_{i,n}^2 \cup \Omega_{i,n}^3.
		\end{aligned}
	\end{multline*}
	Thus, $r_{i,n}$, $0 \leq i \leq \barK$, can be expressed as
	$
		r_{i,n} = r_{i,n}^1 + r_{i,n}^2 + r_{i,n}^3
	$
	with
	\begin{equation*}
		\begin{aligned}
			& r_{i,n}^1 :=\int_\Omega  \1_{ \Omega_{i,n}^1} \bar p \left[ f'_i(\bar y) (y_n -\bar y) - f_{i-1}(y_n) + f_i(\bar y) \right] \dx,\\
			& r_{i,n}^2 :=\int_\Omega \1_{\Omega_{i,n}^2}\bar p \left[ f'_i(\bar y) (y_n -\bar y) - f_{i}(y_n) + f_i(\bar y) \right] \dx,\\
			& r_{i,n}^3 :=\int_\Omega  \1_{\Omega_{i,n}^3}\bar p \left[ f'_i(\bar y) (y_n -\bar y) - f_{i+1}(y_n) + f_i(\bar y) \right] \dx.
		\end{aligned}
	\end{equation*}
	We first estimate $r_{i,n}^2$. For this end, we rewrite it as
	\begin{multline} \label{eq:r-2}
		r_{i,n}^2 = \int_\Omega  \1_{\{\bar y \in (\tau_i, \tau_{i+1}) \}} \bar p \left[ f'_i(\bar y) (y_n -\bar y) - f_{i}(y_n) + f_i(\bar y) \right] \dx \\
		\begin{aligned}[b]
		 	&- \int_\Omega  \1_{\{\bar y \in (\tau_i, \tau_{i+1}), y_n \in (\tau_i-\epsilon_0, \tau_i]  \}} \bar p \left[ f'_i(\bar y) (y_n -\bar y) - f_{i}(y_n) + f_i(\bar y) \right] \dx \\
		 	&- \int_\Omega  \1_{\{\bar y \in (\tau_i, \tau_{i+1}), y_n \in [\tau_{i+1}, \tau_{i+1}+\epsilon_0)  \}} \bar p \left[ f'_i(\bar y) (y_n -\bar y) - f_{i}(y_n) + f_i(\bar y) \right] \dx.
		 \end{aligned}
	\end{multline}
	By using a second-order Taylor expansion of $f_i$ and the limit \eqref{eq:weak-strong-Hadamard-app}, the Lebesgue dominated convergence theorem implies that
	\begin{multline*} \nonumber
		\lim_{n\to\infty} \frac{1}{t_n^2}\int_\Omega  \1_{\{\bar y \in (\tau_i, \tau_{i+1}) \}} \bar p \left[ f'_i(\bar y) (y_n -\bar y) - f_{i}(y_n) + f_i(\bar y) \right] \dx \\
		= - \frac{1}{2}\int_\Omega \1_{\{\bar y \in (\tau_i, \tau_{i+1}) \}}\bar p f''_{yy}(\bar y)\left(S'(\bar u) h\right)^2 \dx.
	\end{multline*} 
	Moreover,  there holds
	\begin{align*}
		\norm{ \1_{\{\bar y \in (\tau_i, \tau_{i+1}), y_n \in (\tau_i-\epsilon_0, \tau_i]  \}} }_{L^1(\Omega)} & \leq \int_\Omega \1_{\{ |\bar y - \tau_i| \leq \epsilon_n \} }  \dx = o(1)
	\end{align*}
	for sufficiently large $n$ with $\epsilon_n := \norm{y_n - \bar y}_{C(\overline\Omega)}$.
	Then 
	the absolute value of the second integral 
	term in the right-hand side of \eqref{eq:r-2} is not greater than
	\begin{align*}
		o(1) \norm{\bar p}_{L^\infty(\Omega)} \norm{f'_i(\bar y) (y_n -\bar y) - f_{i}(y_n) + f_i(\bar y)}_{L^\infty(\Omega)} \leq C\epsilon_n^2 o(1),
	\end{align*}
	which, along with \eqref{eq:G-diff-key}, yields
	\begin{equation} \nonumber
		\lim_{n\to\infty} \frac{1}{t_n^2}\int_\Omega  \1_{\{\bar y \in (\tau_i, \tau_{i+1}), y_n \in (\tau_i-\epsilon_0, \tau_i]  \}} \bar p \left[ f'_i(\bar y) (y_n -\bar y) - f_{i}(y_n) + f_i(\bar y) \right] \dx =0.
	\end{equation} 
	We also have the same limit for the last integral term in the right-hand side of \eqref{eq:r-2}. We then have
	\begin{equation}
		\label{eq:r2-limit}
		\lim_{n\to\infty} \frac{1}{t_n^2}r_{i,n}^2 = - \frac{1}{2}\int_\Omega \1_{\{\bar y \in (\tau_i, \tau_{i+1}) \}}\bar p f''_{yy}(\bar y)\left(S'(\bar u) h\right)^2 \dx.
	\end{equation}	
	We now estimate $r_{i,n}^1$. To this aim, using the identity \eqref{eq:PC1-continuous-cond} gives 
	\begin{multline} \label{eq:r-1}   
		f'_i(\bar y) (y_n -\bar y) - f_{i-1}(y_n) + f_i(\bar y) = -\left[ f_{i-1}(y_n) - f_{i-1}(\tau_i) - f_{i-1}'(\tau_i)(y_n -\tau_i) \right] \\
		\begin{aligned}[b]
			&  -\left[ f_{i}(\tau_i) - f_{i}(\bar y) - f_i'(\bar y)(\tau_i - \bar y) \right] - \left[ f_{i}'(\tau_i) - f_i'(\bar y) \right](y_n- \tau_i)\\
			&  - \left[ f_{i-1}'(\tau_i) - f_i'(\tau_i) \right](y_n- \tau_i).
		\end{aligned}
	\end{multline}
	For a.a. 
	$x \in \Omega_{i,n}^1 =  \{ \bar y \in (\tau_i, \tau_{i+1}), y_n \in (\tau_i-\epsilon_0, \tau_i] \}$, 
	we have 
	\begin{equation*} \label{eq:Omega-i2-set}
		0 \leq \bar y(x) - \tau_i, \tau_i - y_n(x) \leq |y_n(x) - \bar y(x)| \leq \epsilon_n < \epsilon_0
	\end{equation*}
	for  $n$ large enough.
	We thus derive
	\begin{align*}
			\frac{1}{t_n^2}| \1_{ \Omega_{i,n}^{1} }(x)\left[ f_{i-1}(y_n(x)) - f_{i-1}(\tau_i) - f_{i-1}'(\tau_i)(y_n(x) -\tau_i)  \right] | & \leq  \frac{1}{t_n^2}C \1_{ \Omega_{i,n}^{1} }(x) \epsilon_n^2 \\
			 &\leq C \1_{ \Omega_{i,n}^{1} }(x) \to 0
	\end{align*} 
	for a.a. $x \in \Omega$ and for some constant $C$. Analogous limits corresponding to the second and the third terms in the right-hand side of \eqref{eq:r-1} are also validated. We can thus deduce from facts $\1_{ \Omega_{0,n}^{1} } = 0 =  \1_{ \Omega_{K,n}^{3} }$ a.a. in $\Omega$ that
	\begin{equation}
		\label{eq:r1-limit}
		\lim_{n \to \infty} \frac{1}{t_n^2} \{ r_{i,n}^1 + \left[ f_{i-1}'(\tau_i) - f_i'(\tau_i) \right]\int_\Omega \1_{ \Omega_{i,n}^{1} } \bar p (y_n- \tau_i) \dx \} =0, \quad 1 \leq i \leq K
	\end{equation}
	and, analogously,
	\begin{equation}
		\label{eq:r3-limit}
		\lim_{n \to \infty} \frac{1}{t_n^2} \{ r_{i,n}^3 + \left[ f_{i+1}'(\tau_{i+1}) - f_i'(\tau_{i+1}) \right]\int_\Omega \1_{ \Omega_{i,n}^{3} } \bar p (y_n- \tau_{i+1}) \dx \} =0,  0 \leq i \leq K-1.
	\end{equation}
	From \eqref{eq:r2-limit}, \eqref{eq:r1-limit}, and \eqref{eq:r3-limit}, we have from the definition of $\zeta$ and a detailed computation that
	\begin{equation} \nonumber
		\frac{1}{t_n^2}r(y_n) - \frac{1}{t_n^2} \int_\Omega \bar p \zeta(\bar u; t_n, h_n) \dx= - \frac{1}{2}\int_\Omega \1_{\{\bar y \notin E_f \}}\bar p f''_{yy}(\bar y)\left(S'(\bar u) h\right)^2 \dx + o(1),
	\end{equation} 
	which is the desired conclusion.
\end{proof}

We now provide some notion of the second-order generalized derivative for $F$, which was introduced in \cite[Def.~3.2]{WachsmuthWachsmuth2022} to deal with the second derivative for the convex integral functional. 
\begin{definition}[strong second subderivative] \label{def:weaksecond-subderivative}
	Let $X$ be the dual of a separable Banach space $Y$, i.e., $X =Y^*$ and let $G: Y \to (-\infty, +\infty]$ be an extended functional. Assume that $u \in \textrm{dom}(G)$ and $w \in X$. Then the \emph{strong second subderivative} $G''(u,w;\cdot): Y \to [-\infty, + \infty]$ of $G$ at $u$ for $w$ is defined via
	\begin{equation} \nonumber
		G''(u,w;h) := \inf \left\{ \liminf\limits_{n \to \infty} \frac{G(u+ t_n h_n) - G(u) - t_n \langle w, h_n \rangle}{t_n^2/2} \mid t_n \to 0^+, h_n \to h \right\}.
	\end{equation} 
\end{definition}

As a consequence of \cref{lem:2nd-Taylor} and the structural assumption \eqref{eq:structure-non-diff}, the non-smooth curvature functional $Q$ of $F$ at $\bar u$ is identical to the strong second subderivative of $F$ at $\bar u$, as shown in the following.
\begin{proposition}
	\label{prop:Q-weak2subder}
	Let $X = Y^*$ with $Y := L^\barr(\Omega)$ and let  \eqref{eq:structure-non-diff} be fulfilled. Then, there holds 
	\begin{equation} \nonumber
		Q(\bar u, \bar p; h) = F''(\bar u, w; h)
	\end{equation} 
	for all $h \in L^\barr(\Omega)$ with $w := (\bar p + \nu \bar u) =  F'(\bar u) \in X$.
\end{proposition}
\begin{proof}
	Let $\{t_n\} \in c_0^+$ and $\{h_n\} \subset L^\barr(\Omega)$ be arbitrary such that $h_n \to h$ in $L^\barr(\Omega)$.  
	Setting $y_n :=S(\bar u + t_nh_n)$ and applying \cref{lem:2nd-Taylor,lem:key-limit}, we derive for $n$ large enough that
	\begin{multline*}
		F(\bar u + t_n h_n) - F(\bar u) - t_n \int_\Omega w h_n \dx \\
		\begin{aligned}
		 	&=  \int_\Omega \int_0^1 (1-s)\frac{\partial^2 L}{\partial y^2}(x, \bar y + s(y_n -\bar y))(y_n - \bar y)^2 \ds \dx + \frac{\nu}{2} t_n^2 \norm{h_n}^2_{L^2(\Omega)} + r(y_n) \\
		 	& = \int_\Omega \int_0^1 (1-s)\frac{\partial^2 L}{\partial y^2}(x, \bar y + s(y_n -\bar y))(y_n - \bar y)^2 \ds \dx + \frac{\nu}{2} t_n^2 \norm{h_n}^2_{L^2(\Omega)} \\
		 	& \qquad \qquad \qquad - \frac{1}{2} t_n^2\int_\Omega \1_{\{\bar y \notin E_f \}}\bar p f''_{yy}(\bar y)\left(S'(\bar u) h\right)^2\dx  + \int_\Omega  \bar p \zeta(\bar u; t_n, h_n)\dx + o(t_n^2).
		 \end{aligned}
	\end{multline*}
	Dividing the above equation by $t^2_n/2$, 
	taking the limit inferior in the obtained result, 
	and using \eqref{eq:weak-strong-Hadamard-app}, we deduce from the definition of
	$\underline{Q}$  that 
	\begin{multline*}
		\liminf\limits_{n \to \infty} \frac{1}{t_n^2/2}[F(\bar u + t_n h_n) - F(\bar u) - t_n \int_\Omega w h_n \dx] -  \int_\Omega[ \frac{\partial^2 L}{\partial y^2}(x, \bar y) \left(S'(\bar u)h\right)^2 + \nu h^2 ] \dx \\
		\begin{aligned}[t]
			&= 
			 - \int_\Omega \1_{\{ \bar y \notin E_f \}} \bar p f''_{yy}( \bar y)\left(S'( \bar u) h\right)^2 \dx +  2\underline{Q}(\bar u, \bar p; \{t_n\}, \{h_n\}) \\
			& =  - \int_\Omega \1_{\{ \bar y \notin E_f \}} \bar p f''_{yy}( \bar y)\left(S'( \bar u) h\right)^2 \dx +  2\underline{Q}(\bar u, \bar p; \{t_n\}, \{h\}), 
		\end{aligned} 
	\end{multline*}
 	where we have just exploited \cref{lem:invariant} to get the last identity. The definition of strong second subderivative of $F$ and of the non-smooth curvature functional $Q$ thus show the desired identity.
\end{proof}

\medskip

In order to derive the second-order necessary and sufficient optimality conditions for \eqref{eq:P}, we follow \cite{Casas2012,CasasHerzogWachsmuth2012} and introduce the following cone of critical directions: 
\begin{equation}
	\label{eq:critical-cone}
	\mathcal{C}_{L^\barr(\Omega)}(U_{ad}; \bar u) := \{v \in L^\barr(\Omega) \mid v \, \text{satisfies \eqref{eq:critical-direction-1} and } \int_\Omega(\bar p + \nu \bar u) v \dx + \kappa j'(\bar u; v) = 0 \}
\end{equation}
with
\begin{equation}
	\label{eq:critical-direction-1}
	v(x) \left\{
	\begin{aligned}
		& \geq 0 && \text{ for a.a. } x \in \{ \bar u = \alpha \},\\
		& \leq 0 && \text{ for a.a. } x \in \{ \bar u = \beta \}.
	\end{aligned}
	\right.
\end{equation}
Here $\bar u$ is a given admissible control for which there exist $\bar p \in H^1(\Omega) \cap C(\overline\Omega)$ and $\bar \lambda \in \partial j(\bar u)$ satisfying \eqref{eq:1st-OS} for $\chi(\cdot) = \1_{\{ \bar y \notin E_f \}}(\cdot) f'(\bar y(\cdot))$. Recall from \cref{cor:projections} that $\bar p$ and $\bar \lambda$ are uniquely determined when $\lambda^N(\{ \bar y  \in E_f\}) =0$.
\begin{proposition}[{cf. \cite[Prop.~3.4]{Casas2012}}]
	\label{prop:critical-cone-closed-convex}
	Under the condition that $\lambda^N(\{ \bar y  \in E_f\}) =0$, the set $\mathcal{C}_{L^\barr(\Omega)}(U_{ad}; \bar u)$ is a closed and convex cone in $L^\barr(\Omega)$.
\end{proposition}
\begin{proof}
	Obviously, $\mathcal{C}_{L^\barr(\Omega)}(U_{ad}; \bar u)$ is a closed cone in $L^\barr(\Omega)$. We now show the convexity. To that purpose, let $v_1, v_2 \in \mathcal{C}_{L^\barr(\Omega)}(U_{ad}; \bar u)$ and $t \in (0,1)$ be arbitrary and let us set $v := tv_1 + (1-t)v_2$. From the convexity of $j'(\bar u; \cdot)$ due to \eqref{eq:j-dir-der}, one has
	\begin{multline*}
		\int_\Omega(\bar p + \nu \bar u) v \dx + \kappa j'(\bar u; v)  \leq t[\int_\Omega(\bar p + \nu \bar u) v_1 \dx \\
		 +\kappa j'(\bar u; v_1) ] + (1-t)[\int_\Omega(\bar p + \nu \bar u) v_2 \dx +\kappa j'(\bar u; v_2) ] = 0.
	\end{multline*}
	Thanks to \eqref{eq:norm-OS-equi} and \eqref{eq:tangent-cone}, we deduce that
	$
		\int_\Omega(\bar p + \nu \bar u + \kappa \bar \lambda) v \dx \geq 0,
	$
	which together with \eqref{eq:subdiff-directional-relation} gives 
	$
		\int_\Omega(\bar p + \nu \bar u) v \dx + \kappa j'(\bar u; v) \geq 0.
	$
	We thus have $\int_\Omega(\bar p + \nu \bar u) v \dx + \kappa j'(\bar u; v) = 0$ and then $v \in \mathcal{C}_{L^\barr(\Omega)}(U_{ad}; \bar u)$. 
\end{proof}

From now on, let us set
\begin{equation} \label{eq:critical-d}
	\bar d: = \bar p + \nu \bar u + \kappa \bar \lambda.
\end{equation}
By \eqref{eq:normal-OS} and \eqref{eq:tangent-cone}, the following implications hold
\begin{equation*}
	\left\{
	\begin{aligned}
		\bar u(x) = \alpha & \implies & \bar d(x) \geq 0,\\
		\bar u(x) = \beta & \implies & \bar d(x) \leq 0,\\
		\alpha < \bar u(x) < \beta & \implies & \bar d(x) = 0
	\end{aligned}
	\right. \quad \text{and} \quad \left\{
	\begin{aligned}
		\bar d(x) > 0 & \implies & \bar u(x) = \alpha,\\
		\bar d(x) < 0 & \implies & \bar u(x) = \beta 
	\end{aligned}
	\right.
\end{equation*}
for a.a. $x \in \Omega$. From this and the fact from \eqref{eq:subdiff-directional-relation}, \eqref{eq:normal-OS}, and \eqref{eq:tangent-cone}   that
\begin{equation} \nonumber
	0 = \int_\Omega(\bar p + \nu \bar u) v \dx + \kappa j'(\bar u; v) \geq \int_\Omega(\bar p + \nu \bar u) v \dx + \kappa  \int_\Omega \bar\lambda v \dx = \int_\Omega \bar d v\dx \geq  0
\end{equation} 
for all $v \in \mathcal{C}_{L^\barr(\Omega)}(U_{ad}; \bar u)$,
we have for a.a. $x \in \Omega$ and for all $v \in \mathcal{C}_{L^\barr(\Omega)}(U_{ad}; \bar u)$ that
\begin{equation} \label{eq:d-v-orthogonal}
	 \bar d(x) v(x) = 0 \quad \text{and} \quad  j'(\bar u; v) = \int_\Omega \bar \lambda v \dx.  
\end{equation}
Combining the last identity with \eqref{eq:j-dir-der} and \eqref{eq:subderivative-j-formulation} yields
\begin{equation} \nonumber
	\int_{\{ \bar u =0 \} } \bar \lambda v \dx = \int_{\{ \bar u =0 \} } |v| \dx,
\end{equation} 
which, along with the fact that $\bar \lambda(x) \in [-1,1]$ for a.a. $x \in \{ \bar u =0 \}$, leads to 
\begin{equation}
	\label{eq:u-0-set-identity}
	\bar \lambda(x)v(x) = |v(x)| \quad \text{ for a.a. } x \in \{ \bar u =0 \} \, \text{and for all } v \in \mathcal{C}_{L^\barr(\Omega)}(U_{ad}; \bar u).
\end{equation}

The following lemma plays an important role in proving the second-order necessary optimality conditions for \eqref{eq:P}.
\begin{lemma}
	\label{lem:key-critical-dense}
	For any $v \in \mathcal{C}_{L^\barr(\Omega)}(U_{ad}; \bar u)$ and $0< \epsilon < \min\{1, \frac{1}{\beta - \alpha}\}$, there exists a  $v_\epsilon \in L^\barr(\Omega)$ such that the following assertions are fulfilled:
	\begin{center}
		\begin{varwidth}{\textwidth}
			\begin{tasks}[label={(\roman*)},label-width={0.5cm}](2)
				\task \label{item:density} $v_\epsilon \to v$ in $L^\barr(\Omega)$ as $\epsilon \to 0^+$;
				\task \label{item:vk-admissible} $\bar u + t v_\epsilon \in U_{ad}$;
				\task \label{item:vk-d-vanish} $v_\epsilon \bar d = 0$ a.a. in $\Omega$; 
				\task \label{item:vk-critical} $v_\epsilon \in \mathcal{C}_{L^\barr(\Omega)}(U_{ad}; \bar u)$;
				\task \label{item:vk-j-linear} $j(\bar u + t v_\epsilon) - j(\bar u) = t \int_\Omega \bar \lambda v_\epsilon \dx$
			\end{tasks}
		\end{varwidth}
	\end{center}
	for all $0 < t < \epsilon^3$, where $\bar d$ is defined as in \eqref{eq:critical-d}.
\end{lemma}
\begin{proof}
	For any $0< \epsilon < \min\{1, \frac{1}{\beta - \alpha}\}$, we define the function $v_\epsilon$ as follows
	\[
		v_\epsilon(x):= \proj_{[-1/\epsilon^2,+1/\epsilon^2]}(v(x)) 
	\]
	whenever $\bar u(x) + \epsilon v(x) \in [\alpha, \beta]$ and $(|\bar u(x)| = 0 \, \text{or } |\bar u(x)| \geq \epsilon)$; $v_\epsilon(x) := 0$, otherwise. 
	Obviously, $|v_\epsilon(x)| \leq |v(x)|$ for a.a. $x \in \Omega$ and $v_\epsilon \to v$ a.a. in $\Omega$ as $\epsilon \to 0^+$.
	The limit in assertion \ref{item:density} thus follows from Lebesgue's dominated convergence theorem. 
	Besides, we deduce for all $0< t <\epsilon^3 < \epsilon$ that $\bar u(x) + tv_\epsilon(x)$ belongs to
	\begin{equation*} \nonumber
		 \left\{
		\begin{aligned}
		& [\bar u(x), \bar u(x) + tv(x)]  && \text{if } \bar u(x) + \epsilon v(x) \in [\alpha, \beta], v(x) \geq 0 \, \text{and } (|\bar u(x)| = 0 \, \text{or } |\bar u(x)| \geq \epsilon),\\
		& [\bar u(x) + tv(x), \bar u(x)]   && \text{if } \bar u(x) + \epsilon v(x) \in [\alpha, \beta], v(x) < 0 \, \text{and } (|\bar u(x)| = 0 \, \text{or } |\bar u(x)| \geq \epsilon),\\
		& \{\bar u(x)\} && \text{otherwise},
		\end{aligned}
		\right.
	\end{equation*} 
	which shows the inclusion in assertion \ref{item:vk-admissible}.
	Moreover, thanks to the first relation in \eqref{eq:d-v-orthogonal}, we have the identity in assertion \ref{item:vk-d-vanish}. 
	To prove assertions \ref{item:vk-critical} and  \ref{item:vk-j-linear}, we need to show the following identity
	\begin{equation}
		\label{eq:u-0-set-identity-auxi}
		\bar \lambda(x)v_\epsilon(x) = |v_\epsilon(x)| \quad \text{ for a.a. } x \in \{ \bar u =0 \}.
	\end{equation}
	To that end, for a.a. $x \in \{ \bar u =0\}$, we have from the definition of $v_\epsilon$ that
	\begin{equation} \nonumber
		v_\epsilon(x) = \left\{
		\begin{aligned}
		& \proj_{[-1/\epsilon^2,+1/\epsilon^2]}(v(x)) && \text{if } \epsilon v(x) \in [\alpha, \beta],\\
		& 0 && \text{otherwise}.
		\end{aligned}
		\right.
	\end{equation} 
	Since $\frac{1}{\epsilon}[\alpha, \beta] \subset [-1/\epsilon^2,+1/\epsilon^2]$ for $0< \epsilon < \frac{1}{\beta-\alpha}$, then there holds for a.a. $x \in \{ \bar u =0\}$ that
	\begin{equation*}
		v_\epsilon(x) = \left\{
		\begin{aligned}
		& v(x) && \text{if } \epsilon v(x) \in [\alpha, \beta],\\
		& 0 && \text{otherwise},
		\end{aligned}
		\right.	
	\end{equation*}
	which, along with \eqref{eq:u-0-set-identity}, yields \eqref{eq:u-0-set-identity-auxi}.	
	Combining \eqref{eq:u-0-set-identity-auxi} with \eqref{eq:j-dir-der} and \eqref{eq:subderivative-j-formulation} yields
	$
		j'(\bar u; v_\epsilon) = \int_\Omega \bar \lambda v_\epsilon \dx,
	$
	which, 
	in combination with 
	the identity in \ref{item:vk-d-vanish}, gives 
	\begin{equation} \nonumber
		\int_\Omega (\bar p + \nu \bar u)v_\epsilon \dx + \kappa j'(\bar u; v_\epsilon) = \int_\Omega \bar d v_\epsilon \dx = 0.
	\end{equation} 
	On the other hand, we deduce from the definition of $v_\epsilon$ and from the fact $v$ satisfies \eqref{eq:critical-direction-1} that $v_\epsilon \geq 0$ a.a. in $\{ \bar u = \alpha \}$ and  $v_\epsilon \leq 0$ a.a. in $\{ \bar u = \beta \}$.
	Then, the inclusion in assertion \ref{item:vk-critical} is verified.
	For \ref{item:vk-j-linear}, we first analyze the case where for a.a. $x \in \{\bar u>0 \}$. In this situation, there holds
	\begin{equation} \nonumber
		v_\epsilon(x) = \left\{
		\begin{aligned}
		& \proj_{[-1/\epsilon^2,+1/\epsilon^2]}(v(x))  && \text{if } \bar u(x) + \epsilon v(x) \in [\alpha, \beta] \, \text{and } \bar u(x) \geq \epsilon,\\
		& 0 && \text{otherwise}.
		\end{aligned}
		\right.
	\end{equation} 
	This implies for all $0 < t < \epsilon^3$ that
	$
	 	\bar u(x) + t v_\epsilon(x) > \min \{0, \epsilon + \epsilon^3 \frac{-1}{\epsilon^2} \} = 0.
	$
	Combining this with the fact that $\bar \lambda(x) = 1$ for a.a. $x \in \{\bar u>0 \}$ gives 
	\begin{equation} \nonumber
		|\bar u(x) + t v_\epsilon(x)| - |\bar u(x)| = \bar u(x) + t v_\epsilon(x) - \bar u(x) = t \bar\lambda(x) v_\epsilon(x).
	\end{equation} 
	Analogously, there holds $|\bar u(x) + t v_\epsilon(x)| - |\bar u(x)| = t \bar\lambda(x) v_\epsilon(x)$ for a.a. $x \in \{\bar u<0 \}$ and for all $0 < t < \epsilon^3$. From these and \eqref{eq:u-0-set-identity-auxi}, we deduce that
	\begin{equation} \nonumber
		|\bar u(x) + t v_\epsilon(x)| - |\bar u(x)| = t \bar \lambda(x) v_\epsilon(x)	
	\end{equation} 
	for a.a. $x \in \Omega$ and for all $0 < t < \epsilon^3$. Assertion \ref{item:vk-j-linear} then follows.
\end{proof}

We  are now  ready to state the second-order optimality conditions for \eqref{eq:P} in terms of the curvature functional $Q$ defined in \cref{def:curvature-func-app}. 
\begin{theorem}[second-order necessary optimality condition] 
	\label{thm:2nd-OS-nec}
	Let $\bar u$ be a local minimizer of \eqref{eq:P} such that the associated state $\bar y:= S(\bar u)$ fulfills the structural assumption \eqref{eq:structure-non-diff}. 
	{
	Then there exist a unique adjoint state $\bar p \in H^1_0(\Omega) \cap C(\overline\Omega)$ and a unique multiplier $\bar \lambda \in \partial j(\bar u)$ satisfying
}%
	\eqref{eq:1st-OS} for $\chi(\cdot) = \1_{\{ \bar y \notin E_f \}}(\cdot) f'(\bar y(\cdot))$. Moreover, the following second-order necessary optimality condition holds:
	\begin{equation}
		\label{eq:2nd-OS-nec}
		Q(\bar u, \bar p;h)  \geq 0 \qquad\text{for all }h\in \mathcal{C}_{L^\barr(\Omega)}({U}_{ad};\bar u).
	\end{equation}
\end{theorem}
\begin{proof}
	Thanks to \eqref{eq:structure-non-diff}, there holds $\lambda^N(\{\bar y  \in E_f \}) = 0$ and thus $S$ is G\^{a}teaux-differentiable in $\bar u$ as a result of \cref{prop:G-diff-control2state}.
	In view of  \cref{thm:1st-OS} and \cref{cor:projections}, $\bar p \in H^1_0(\Omega) \cap C(\overline\Omega)$ and $\bar \lambda \in \partial j(\bar u)$ exist uniquely and satisfy \eqref{eq:1st-OS}. 
	It remains to prove \eqref{eq:2nd-OS-nec}. To do this, let $h \in \mathcal{C}_{L^\barr(\Omega)}({U}_{ad};\bar u)$ and $\{t_n\} \in c_0^+$ be arbitrary but fixed. It suffices to show that
	\begin{multline}
		\label{eq:2nd-OS-nec-2}
		\int_\Omega[ \frac{\partial^2 L}{\partial y^2}(x, \bar y) \left(S'( \bar u) h\right)^2 + \nu h^2 ] \dx 
		- \int_\Omega \1_{\{ \bar y \notin E_f \}}  \bar p f''_{yy}( \bar y)(S'(\bar u) h)^2 \dx \\
		 +  2\underline{Q}(\bar u,\bar p;\{t_n\},\{h\}) \geq 0.
	\end{multline}    
	In order to get \eqref{eq:2nd-OS-nec-2}, by \eqref{eq:key-term-sn}, a subsequence $\{t_{n_k}\}$ of $\{t_n\}$ exists and fulfills
	\begin{equation*}
		\underline{Q}(\bar u,\bar p;\{t_n\},\{h\}) =  \lim_{k \to \infty} \frac{1}{t_{n_k}^2} \int_\Omega \bar p \zeta(\bar u; t_{n_k}, h)  \dx.
	\end{equation*}
	Applying  \cref{lem:key-critical-dense} for $\epsilon := (2t_{n_k})^{1/3} < \min\{1, \frac{1}{\beta-\alpha}\}$ with sufficiently large $k$ yields the existence of an $h_k \in \mathcal C_{L^\barr(\Omega)}(U_{ad}, \bar u)$ such that
	\begin{equation*}
		h_k \to h \quad \text{in } L^\barr(\Omega), \quad  h_k \bar d = 0 \quad \text{a.a. in } \Omega, \quad \bar u + t_{n_k}h_k \in U_{ad},  	
	\end{equation*}
	and 
	\[
		j(\bar u + t_{n_k}h_k) - j(\bar u) = t_{n_k} \int_\Omega \bar \lambda h_k \dx.
	\]
	By setting $u_k:= \bar u + t_{n_k}h_k$ and $y_k := S(\bar u + t_{n_k}h_k)$ and using the optimality of $\bar u$, we then deduce from \cref{lem:2nd-Taylor} that
	\begin{multline*}
		0  \leq J_\kappa(u_k) - J_\kappa(\bar u) = F(u_k) - F(\bar u) + \kappa (j(u_k) - j(\bar u)) \\
		\begin{aligned}
			&= \int_\Omega\int_0^1 (1-s) \frac{\partial^2 L}{\partial y^2}(x,\bar y+ s(y_k -\bar y))(y_k - \bar y)^2 \ds \dx \\
			& \qquad + \frac{\nu}{2} t_{n_k}^2 \norm{h_k}_{L^2(\Omega)}^2 + t_{n_k} \int_\Omega (\bar p + \nu \bar u) h_k \dx + t_{n_k} \kappa \int_\Omega \bar \lambda h_k \dx + r(y_k) \\
			& = \int_\Omega\int_0^1 (1-s) \frac{\partial^2 L}{\partial y^2}(x,\bar y+ s(y_k -\bar y))(y_k - \bar y)^2 \ds \dx \\
			& \qquad \qquad  + \frac{\nu}{2} t_{n_k}^2 \norm{h_k}_{L^2(\Omega)}^2 +t_{n_k} \int_\Omega \bar d h_k \dx +r(y_k)\\
			& = \int_\Omega\int_0^1 (1-s) \frac{\partial^2 L}{\partial y^2}(x,\bar y+ s(y_k -\bar y))(y_k - \bar y)^2 \ds \dx + \frac{\nu}{2} t_{n_k}^2 \norm{h_k}_{L^2(\Omega)}^2  +r(y_k) 
		\end{aligned}
	\end{multline*}
	for $k$ large enough. Dividing the above inequality by $t_{n_k}^2/2$ and 
	{
	taking the limit inferior in the obtained result, 
}%
	we therefore derive \eqref{eq:2nd-OS-nec-2} from \cref{prop:control-to-state}, \cref{lem:key-limit,lem:invariant}.
\end{proof}
{
	For the second-order sufficient optimality conditions, we require that $\barr =2$ and thus the dimension $N$ needs to be not greater than three, due to \eqref{eq:barr-choice}.
}%
\begin{theorem}[second-order sufficient optimality conditions] 
	\label{thm:2nd-OS-suf}
	{
	Assume that $N \in \{1,2,3\}$ and that
}%
	$\bar u$ is an admissible point of \eqref{eq:P} for which \eqref{eq:structure-non-diff} is satisfied by $\bar y:= S(\bar u)$. Assume further that there exist an adjoint state $\bar p \in H^1_0(\Omega) \cap C(\overline\Omega)$ and a multiplier $\bar \lambda \in \partial j(\bar u)$ satisfying \eqref{eq:1st-OS} for $\chi(\cdot) = \1_{\{ \bar y \notin E_f \}}(\cdot) f'(\bar y(\cdot))$. If the following second-order sufficient condition is fulfilled:
	\begin{equation}
		\label{eq:2nd-OS-suff}
		Q(\bar u,\bar p;h) >0
		\qquad\text{for all }h\in \mathcal{C}_{L^2(\Omega)}({U}_{ad};\bar u)\setminus \{0\},
	\end{equation}
	then constants $c, \rho >0$ exist and satisfy
	\begin{equation}
		\label{eq:quadratic-grownth}
		J_\kappa(\bar u) + c \norm{u - \bar u}_{L^2(\Omega)}^2 \leq J_\kappa(u) \quad \text{for all } u \in {U}_{ad} \cap \overline B_{L^2(\Omega)}(\bar u, \rho).
	\end{equation}
	Particularly, $\bar u$ is a strict local minimizer of \eqref{eq:P}.
\end{theorem}
\begin{proof}
	We first note that $S$ is G\^{a}teaux-differentiable in $\bar u$ as a result of \cref{prop:G-diff-control2state} and \eqref{eq:structure-non-diff}.	
	We  prove the theorem by a contradiction argument. Suppose the claim was false. Then $\{u_n\} \subset \mathcal{U}_{ad}$ exists and fulfills for all $n\geq 1$ that
	\begin{equation} \label{eq:contradiction}
		\norm{u_n - \bar u}_{L^2(\Omega)} < \frac{1}{n} \quad \text{and} \quad F(\bar u) + \kappa j(\bar u) + \frac{1}{n}\norm{u_n - \bar u}_{L^2(\Omega)}^2 > F(u_n) + \kappa j(u_n). 
	\end{equation}
	Setting $t_n := \norm{u_n - \bar u}_{L^2(\Omega)}$ and $h_n := \frac{u_n - \bar u}{t_n}$ yields $\norm{h_n}_{L^2(\Omega)} = 1$ and $\{t_n\} \in c_0^+$. 
	Moreover, by using a subsequence if necessary, there holds $h_n \rightharpoonup h$ in $L^2(\Omega)$ for some $h \in \textrm{cl}_{L^2(\Omega)}(\textrm{cone}({U}_{ad}-\bar u))$.
	Obviously, \eqref{eq:critical-direction-1} is satisfied by $v:=h_n$ and thus by $v:=h$. 
	We now show that 
	\begin{equation}
		\label{eq:h-critical-dir}
		\int_\Omega(\bar p + \nu \bar u) h \dx + \kappa j'(\bar u; h) = 0,
	\end{equation}
	which leads to $h \in \mathcal{C}_{L^2(\Omega)}({U}_{ad};\bar u)$. To this end, by using the last inequality in \eqref{eq:contradiction}, there holds
	\begin{equation} \nonumber
		F(u_n) - F(\bar u) + \kappa (j(u_n) - j(\bar u)) < \frac{1}{n}t_n^2,
	\end{equation} 
	or, equivalently
	\begin{multline} \label{eq:contraction-auxi}
		 \int_\Omega\int_0^1 (1-s) \frac{\partial^2 L}{\partial y^2}(x,\bar y + s(y_n -\bar y))(y_n - \bar y)^2 \ds \dx  + \frac{\nu}{2}t_n^2 \norm{h_n}_{L^2(\Omega)}^2\\
		\begin{aligned}[t]
		&+ t_n \int_\Omega (\bar p + \nu \bar u) h_n \dx  + r(y_n) + \kappa (j(\bar u + t_n h_n) - j(\bar u)) < \frac{1}{n}t_n^2
		\end{aligned}
	\end{multline}
	 with $y_n := S(u_n)$, due to \cref{lem:2nd-Taylor}. Dividing this by $t_n$, using the fact that $\norm{h_n}_{L^2(\Omega)}=1$, and 
	 {
	 taking to the limit inferior in the obtained result, 
	}%
	 we conclude from \cref{prop:control-to-state}, \cref{lem:key-limit}, and \eqref{eq:j-weak-dir-deri} that $\int_\Omega(\bar p + \nu \bar u) h \dx + \kappa j'(\bar u; h) \leq 0.$
	From this, \eqref{eq:subdiff-directional-relation} and \cref{rem:normal-OS}, we derive
	\begin{equation} \nonumber
		0 \geq \int_\Omega(\bar p + \nu \bar u) h \dx + \kappa j'(\bar u; h) \geq \int_\Omega(\bar p + \nu \bar u) h \dx + \kappa \int_\Omega \bar\lambda h \dx  \geq 0,
	\end{equation} 
	which proves \eqref{eq:h-critical-dir}. 
	On the other hand, combining \eqref{eq:subdiff-directional-relation} with \eqref{eq:normal-OS} yields
	\begin{equation} \label{eq:contraction-first-order-part}
		\frac{t_n \int_\Omega (\bar p + \nu \bar u) h_n \dx   + \kappa (j(\bar u + t_n h_n) - j(\bar u)) }{t_n} \geq \int_\Omega(\bar p + \nu \bar u) h_n \dx + \int_\Omega \kappa\bar \lambda h_n \dx \geq 0.
	\end{equation}
	Dividing \eqref{eq:contraction-auxi} by $t_n^2$ and using \eqref{eq:contraction-first-order-part} and the fact that $\norm{h_n}_{L^2(\Omega)} =1$, we arrive at
	\begin{equation} \nonumber
		\int_\Omega\int_0^1 (1-s) \frac{\partial^2 L}{\partial y^2}(x,\bar y + s(y_n -\bar y))\frac{(y_n - \bar y)^2}{t_n^2} \ds \dx + \frac{\nu}{2}  + \frac{1}{t_n^2} r(y_n)  < \frac{1}{n}.
	\end{equation} 
	Taking the limit inferior and employing \cref{prop:control-to-state} as well as \cref{lem:key-limit} yield
	\begin{multline*} \nonumber
		\frac{1}{2}\int_\Omega \frac{\partial^2 L}{\partial y^2}(x,\bar y )(S'(\bar u)h)^2  \dx + \frac{\nu}{2} -\frac{1}{2}\int_\Omega \1_{\{\bar y \notin E_f \}}\bar p f''_{yy}(\bar y)\left(S'(\bar u) h\right)^2 \dx \\ + \underline{Q}(\bar u,\bar p; \{t_n\}, \{h_n\}) \leq 0.
	\end{multline*} 
	In view of \cref{lem:invariant}, there holds
	\begin{multline*}
		 \frac{1}{2}\int_\Omega \frac{\partial^2 L}{\partial y^2}(x,\bar y )(S'(\bar u)h)^2  \dx + \frac{\nu}{2} \norm{h}^2_{L^2(\Omega)} -\frac{1}{2}\int_\Omega \1_{\{\bar y \notin E_f \}}\bar p f''_{yy}(\bar y)\left(S'(\bar u) h\right)^2 \dx  \\
		 + \underline{Q}(\bar u,\bar p; \{t_n\}, \{h\})
		+ \frac{\nu}{2} (1 - \norm{h}^2_{L^2(\Omega)}) \leq 0.
	\end{multline*}
	The definition of $Q$ in \cref{def:curvature-func-app} then implies that
	\begin{equation}  \label{eq:contradiction2}
		0 \geq \frac{1}{2}Q(\bar u, \bar p;h)+  \frac{\nu}{2}(1 -\norm{h}_{L^2(\Omega)}^2).
	\end{equation}
	On the other hand, we have $\norm{h}_{L^2(\Omega)} \leq 1$ according to the fact that  $\norm{h_n}_{L^2(\Omega)} = 1$ and the weak lower semincontinuity of the norm in $L^2(\Omega)$.
	From this, \eqref{eq:contradiction2}, and \eqref{eq:2nd-OS-suff}, we deduce that $h = 0$. Inserting $h=0$ into \eqref{eq:contradiction2} gives $0 \geq \frac{\nu}{2} >0$. This is impossible.
\end{proof}

\subsection{An explicit formulation of $Q$ and second-order optimality conditions in the explicit form} \label{sec:2ndOS-explicit-form}

Throughout this subsection, we assume that the following assumption is fulfilled.
\begin{assumption3}
	\item \label{ass:structural-restricted}
	{
	The function $\bar y$ belongs to $C^1(\overline\Omega \cap \mathcal{O}) \cap C_0(\overline\Omega)$ and  satisfies the following implication
	\begin{equation}
		\label{eq:structural-restricted}
		\bar y(x) \in E_f \quad \implies \quad |\nabla \bar y(x)| \neq 0 \quad \forall x \in \overline\Omega \cap \mathcal{O}
	\end{equation}
	for some open set $\mathcal{O}$ in $\Rbb^N$ containing $\{\bar y \in E_f \}$.
}%
	Moreover, if $\tau_i =0$ for some $1 \leq i \leq K$ and if the dimension $N \geq 3$, then there holds
	\begin{equation}
		\label{eq:boundary-intersection-empty}
		\textrm{dist}(\partial \Omega, (\{ \bar y =0 \} \backslash \partial \Omega)) > 0.
	\end{equation}
	Here $\textrm{dist}(U,V)$ stands for the usual distance between two sets $U$ and $V$.
\end{assumption3}
The structural assumption \eqref{eq:structural-restricted} automatically implies \eqref{eq:boundary-intersection-empty} for $N=1$ as a result of \cref{prop:accumlation-none} below. This also holds for $N=2$; see, e.g. \cite[Thm.~2.5 \& Lem.~2.16]{AlbertiBianchiniCrippa2013}.
Besides, the assumption \eqref{eq:structural-restricted} shows that \ref{ass:structure-non-diff} is verified; see, e.g.  \cite[Lem.~3.2]{DeckelnichHinze2012}. 
A kind of this assumption, where a differentiable switching function is assumed to be in place of the state $\bar y$, was used in \cite{Felgenhauer2003} and the references therein to deal with the second-order optimality conditions in the bang-bang controls for the control problem governed by ordinary differential equations. 
A close assumption that is the same with \eqref{eq:structural-restricted} but imposed on the adjoint state $\bar p$ instead of the state $\bar y$
was also exploited in \cite[Sec.~6]{ChristofWachsmuth2018} to get a directional Taylor-like expansion for the $L^1(\Omega)$-norm, which helps to establish the explicit formulation of the directional curvature functional for the indicator function of the interval $[-1,1]$ (see, also, \cite{ChristofMeyer2019}); in \cite[Sec.~4]{WachsmuthWachsmuth2022} to derive the strict twice epi-differentiability of the convex integral functional over measures. 

\medskip 

{
	In order to establish the explicit formulation of $Q$ defined in \cref{def:curvature-func-app}, we shall follow the approach used in \cite[Props.~ 4.6 \& 4.8]{ChristofMeyer2019}; see, also, \cite[Lem.~6.10]{ChristofWachsmuth2018} and \cite[Lem.~3.3 \& Thm.~ 3.4]{WachsmuthWachsmuth2022}. We first formulate the associated formulation in the lower-dimensional setting and the obtained result is then lifted to the $N$-dimensional situation via using the standard partition-of-unity argument.
}%
For the one-dimensional case, i.e., $\Omega := (a,b) \subset \Rbb$, the following result shows that the assumption \eqref{eq:structural-restricted} implies the finiteness of the level set $\{ \bar y \in E_f \}$ and thus the condition \eqref{eq:boundary-intersection-empty}. 
\begin{proposition}
	\label{prop:accumlation-none}
	Let $\Omega := (a,b) \subset \Rbb$ for some $a, b \in \Rbb$.
	Assume that \eqref{eq:structural-restricted} is satisfied. Then,  the set $\{ \bar y \in E_f\}$ has finitely many points.
\end{proposition}
\begin{proof}
	Assume that the conclusion is false. Then, there exists a sequence $\{s_n \} \subset \{ \bar y \in E_f \}$.  Using the Bolzano--Weierstrass theorem yields that there is a subsequence $\{s_{n_k} \} \subset \{ \bar y \in E_f \}$ such that
	$ s_{n_k} \to \bar s$ as $k \to \infty$, for some $\bar s \in \Rbb$. 
	{
	The closedness of the set 
}%
	$\{ \bar y \in E_f \}$ infers that $\bar s \in \{ \bar y \in E_f \}$. There thus holds 
	$
		\bar y(\bar s) = \tau = \bar y(s_{n_k}) 
	$
	for some $\tau \in E_f$ and for all $k$ large enough, according to the continuity of $\bar y$ over $\overline\Omega$. We then have
	\begin{equation} \nonumber
		\bar y'(\bar s) = \lim\limits_{k \to \infty} \frac{\bar y(s_{n_k}) - \bar y(\bar s)}{s_{n_k} - \bar s} = 0,
	\end{equation} 
	which contradicts \eqref{eq:structural-restricted}.
\end{proof}

\medskip

The upcoming theorem establishing the explicit formulation of $Q$ is prepared by the next lemma, which studies the one-dimensional situation $N=1$. 
{
	We thank comments made from an anonymous reviewer who has suggested using the $H^1$-weak limits in \eqref{eq:yn-H1-convergence} and \eqref{eq:yn-H1-convergence-general} below, instead of exploiting more rigorous conditions.
}%
\begin{lemma}
	\label{lem:curvature-explicit-form-one-dim}
	Let $\Omega := (a,b) \subset \Rbb$ and let $\{t_n \} \in c_0^+$ be arbitrary. 
	Assume that \eqref{eq:structural-restricted} is fulfilled. 
	{
	Let $y_n \in H^1(a,b)$ be such that
	\begin{equation}
		\label{eq:yn-H1-convergence}
		\frac{y_n - \bar y}{t_n} \rightharpoonup z \quad \text{weakly in} \quad  H^1(a,b)
	\end{equation}
	for some $z \in H^1(a,b)$. Assume further that $\{a, b\} \cap \{\bar y =\tau_i \} = \emptyset$ for some $1 \leq i \leq K$.
}%
	Then, for any $p \in C([a,b])$, there holds for $T_n^i$ given in \eqref{eq:Tn-i} that
	\begin{equation} \label{eq:Q-explicit-one-dim}
		\frac{1}{t_n^2} \int_{a}^b p T_n^i \ds  \to -\frac{1}{2} \int_{ \{ \bar y = \tau_i\} } \frac{p(x)z(x)^2}{| \bar y'(x)|} \dH^{0}(x).
	\end{equation}
\end{lemma}
\begin{proof}
	{
	By putting $z_n := \frac{y_n - \bar y}{t_n}$ and $\eta_n := z_n - z$, we have $z_n, \eta_n \in H^1(a,b)$. Moreover, we deduce from the limit in \eqref{eq:yn-H1-convergence} and the compact embedding $H^1(a,b) \Subset C([a,b])$ (see, e.g., \cite[Thm.~9.16]{Brezis1}) that 
	\begin{equation}
		\label{eq:zn-limit-uniformly}
		\left\{
		\begin{aligned}
			& z_n \,\text{converges to } z, \\
			& \eta_n \,\text{converges to }  0
		\end{aligned}
		\right.
		\quad \text{weakly in } H^1(a,b) \, \text{and strongly in} \, C([a,b]).
	\end{equation}
}%
	Setting $\epsilon_n := \norm{y_n - \bar y}_{C([a,b])}$ thus yields $\epsilon_n \to 0$ as $n \to\infty$. For any $x \in \{\bar y = \tau_i \}$ and $\rho > 0$, we define the following set
	\begin{equation} \nonumber
		\Omega_{i,x}^\rho := \{ s \in [a,b] : |s - x| < \rho \} = (x -\rho, x+ \rho) \cap [a,b].
	\end{equation} 
	Since $\epsilon_n \to 0$,  a $\rho_0$ exists and satisfies 
	\begin{multline*}
		\{ \bar y \in ( \tau_i, \tau_{i+1} ),  y_n \in (\tau_i - \epsilon_0, \tau_i) \} \cup \{ \bar y \in (\tau_{i-1}, \tau_i), y_n \in (\tau_i , \tau_{i} +\epsilon_0) \} \\
		 \subset \{ |\bar y - \tau_i | \leq \epsilon_n \} \subset \bigcup_{x \in \{ \bar y = \tau_i \}} \Omega_{i,x}^{\rho_0}
	\end{multline*}
	for $n$ large enough. From this and \eqref{eq:Tn-i}, $T_n^i$ can be decomposed as
	\begin{multline*} \nonumber
			T_n^i = \sum_{ x \in \{ \bar y = \tau_i \} } \1_{\Omega_{i,x}^{\rho_0} }  (y_n-\tau_i  )\left( \1_{\{\bar y \in (\tau_i, \tau_{i+1}), y_n \in (\tau_i-\epsilon_0, \tau_i) \}} \right. \\
			\left. - \1_{\{ y_n \in (\tau_{i}, \tau_i+\epsilon_0) , \bar y \in (\tau_{i-1},\tau_i) \}} \right) =: \sum_{ x \in \{ \bar y = \tau_i \} } T_{n}^{i,x}.
	\end{multline*} 
	By \cref{prop:accumlation-none}, the level set $\{\bar y = \tau_i\}$ has finitely many points.
	Therefore, in order to show \eqref{eq:Q-explicit-one-dim}, it suffices to prove that
	\begin{equation} \label{eq:Q-explicit-one-dim-local}
		\frac{1}{t_n^2} \int_{a}^b p T_n^{i,x} \ds  \to -\frac{1}{2}  \frac{p(x)z(x)^2}{| \bar y'(x)|} \quad \text{for all } x \in \{ \bar y = \tau_i \}.
	\end{equation}
	To that end, 
	let $x \in \{ \bar y = \tau_i \}$ be arbitrary, but fixed.	
	Without loss of generality (w.l.o.g.) we assume that $\bar y'(x) >0$.	
	By reducing $\rho_0 >0$ small enough and 
	{
	using the limit in \eqref{eq:yn-H1-convergence}, 
}%
	we deduce that there exists a positive constant $m$ satisfying
	{
	\begin{equation}
		\label{eq:structure-positive-local}
		\left\{
		\begin{aligned}
			& y'(s) \geq m > 0 && \text{for all} \, s \in \Omega_{i,x}^{\rho_0},\\
			& y_n'(s) \geq m >0 && \text{for a.a.} \, s \in \Omega_{i,x}^{\rho_0} \, \text{and for $n$ sufficiently large. }
		\end{aligned}
		\right.
	\end{equation}
	The second inequality in \eqref{eq:structure-positive-local} as well as a  representation theorem for Sobolev functions (see, e.g., \cite[Thm~8.2]{Brezis1}) implies that $y_n$ is strictly increasing on $\Omega_{i,x}^{\rho_0}$. Therefore, the set $\{ y_n = \tau_i \} \cap \Omega_{i,x}^{\rho_0}$ has at most one element. 
}%
%
%
%
	Since $x \in (a,b)$, we can assume (by reducing $\rho_0$ if necessary) that $\Omega_{i,x}^{\rho_0} = (x- \rho_0, x+ \rho_0)$ and thus that 
	\begin{equation*}
		\bar y\mid_{(x, x+ \rho_0)} -\tau_i >0 \quad \text{and} \quad \bar y\mid_{(x - \rho_0, x)} - \tau_i < 0.
	\end{equation*}
	
	{
	We now prove that the set $\{ y_n = \tau_i \} \cap \Omega_{i,x}^{\rho_0}$ is a singleton for $n$ large enough. Suppose, by contradiction, that there exists a subsequence, denoted in the same way, such that	
	$\{ y_n = \tau_i \} \cap \Omega_{i,x}^{\rho_0} = \emptyset$ for all $n \geq 1$. We can  w.l.o.g. assume that 
	\[
		y_n(s) > \tau_i \quad \text{for all } s \in \Omega_{i,x}^{\rho_0} = (x- \rho_0, x+ \rho_0) 
	\]
	and for all $n \geq 1$. As a result of \eqref{eq:structure-positive-local}, one has
	$y_n(x - \rho_0) > \tau_i = \bar y(x) > \bar y(x-\rho_0)$. This implies that 
	\[
		\epsilon_n = \norm{y_n - \bar y}_{C([a,b])} \geq y_n(x - \rho_0) - \bar y(x - \rho_0) > \tau_i - \bar y(x - \rho_0) \nrightarrow 0,
	\]
	which contradicts the limit $\epsilon_n \to 0$.
	We have shown that $\{ y_n = \tau_i \} \cap \Omega_{i,x}^{\rho_0}$ is a singleton for $n$ large enough. 
}%

	{
	For $n$ sufficiently large, we denote by $x_n$ the intersection of  $\{ y_n = \tau_i \}$ with $ \Omega_{i,x}^{\rho_0}$.
	Moreover, there holds
	\begin{equation*}
		y_n(s) = \bar y(s) + t_n z(s) + t_n\eta_n(s)  
	\end{equation*}
	for all $s \in [a,b]$ and
	for all $n \geq 1$. The mean value theorem yields 
	\begin{equation} \nonumber
		y_n(x_n) = \bar y(x) + \bar y'(x+ \theta(x_n - x))(x_n -x) + t_n[z(x_n) + \eta_n(x_n)], \quad \theta \in (0,1).
	\end{equation} 
	From this and the fact $\bar y(x) = y_n(x_n) = \tau_i$, we then have for all $n$ large enough that
	$\bar y'(x+ \theta(x_n - x))(x_n -x) = - t_n[z(x_n) + \eta_n(x_n)]$. 
	This, together with the first inequality in \eqref{eq:structure-positive-local}, gives
	\begin{equation}
		\label{eq:xn-x-expression}
		x_n -x = - \frac{t_n}{\bar y'(x+ \theta(x_n - x))} [z(x_n) + \eta_n(x_n)].
	\end{equation}
	Combining this with the limits in \eqref{eq:zn-limit-uniformly} yields $x_n \to x$.
	Dividing  \eqref{eq:xn-x-expression} by $t_n$ and then passing to the limit in the obtained identity, one therefore has
	\begin{equation}
		\label{eq:xn-x-distance}
		\frac{x_n-x}{t_n} \to - \frac{z(x)}{\bar y'(x)} \quad \text{as } n \to \infty,
	\end{equation}	
	thanks to \eqref{eq:zn-limit-uniformly}.
}%
	For large enough $n$, we define the following sets (depending also on $i$):
	\begin{equation} \nonumber
		\left\{
		\begin{aligned}
			\Omega_n^+ &:= \{ \bar y \in ( \tau_i, \tau_{i+1} ),  y_n \in (\tau_i - \epsilon_0, \tau_i) \} \cap \Omega_{i,x}^{\rho_0}, \\
			\Omega_n^{-} & :=  \{ \bar y \in (\tau_{i-1}, \tau_i), y_n \in (\tau_i , \tau_{i} +\epsilon_0) \} \cap \Omega_{i,x}^{\rho_0}.
		\end{aligned}
		\right.
	\end{equation}  	
	We now split the sequence $\{n\}$ into subsequences, still denoted by the same symbol, that satisfy one of the following conditions:
	\begin{multicols}{3}
		\begin{enumerate}[label=(\alph*)]
			\item \label{item:x-equal-xn} $x = x_n$; 
			\item \label{item:x-less-xn} $x < x_n$;
			\item \label{item:x-great-xn} $x >x_n$.
		\end{enumerate}
	\end{multicols}
	For \ref{item:x-equal-xn}, there holds $\Omega_n^+ = \Omega_n^{-} = \emptyset$ and thus $T_n^{i,x} =0$ for $n$ large enough. Moreover, according to  \eqref{eq:yn-H1-convergence}, one has $z(x) = 0$ and we thus have \eqref{eq:Q-explicit-one-dim-local}. For \ref{item:x-less-xn}, we deduce for $n$ large enough that
	$\Omega_n^+ = (x, x_n)$ and $\Omega_n^{-} = \emptyset,
	$
	which infers
	{
	\begin{multline} \label{eq:Tn-i-decom}
		\int_{a}^b p T_n^{i,x} \ds  = \int_{a}^b p(y_n - \tau_i) (\1_{\Omega_n^+} - \1_{\Omega_n^{-}}) \ds =  \int_{x}^{x_n} p(s)(y_n(s) - \tau_i) \ds \\
		\begin{aligned}[b]
			& = \int_{x}^{x_n} p(s) \left\{ [\bar y(s) - \tau_i] + t_n [z(s) +  \eta_n(s)] \right\} \ds \\
			& = \int_{x}^{x_n}  p(s) \int_x^{s} \bar y'(t)dt \ds  + t_n \int_{x}^{x_n} p(s) [z(s) +  \eta_n(s)]  \ds \\
			& = \left[\int_{x}^{x_n}  p(s) \int_x^{s} \bar y'(t)dt \ds + t_n \int_{x}^{x_n} p(s) z(s) \ds \right] + t_n\int_{x}^{x_n} p(s) \eta_n(s)  \ds \\
			& =: e_n^1 + e_n^2. 
		\end{aligned}
	\end{multline}
	Obviously, one has
	\begin{align*}
		|e_n^2| & \leq t_n \int_{x}^{x_n}  |p(s)| \norm{ \eta_n}_{C([a,b])} \ds  \leq  \norm{p}_{L^\infty(a,b)} \norm{\eta_n}_{C([a,b])} t_n |x_n -x|,
	\end{align*}
}%
	which, along with \eqref{eq:xn-x-distance} and the second limit in \eqref{eq:zn-limit-uniformly}, yields
	$
		\frac{1}{t_n^2}e_n^2 \to 0.
	$
	Moreover, by exploiting \eqref{eq:xn-x-distance} and the continuity over $[a,b]$ of $p$, a simple computation gives
	\begin{equation} \nonumber
		\frac{1}{t_n^2}e_n^1 \to - \frac{p(x)z(x)^2}{2\bar y'(x)} = - \frac{p(x)z(x)^2}{2|\bar y'(x)|} \quad  \text{as } n \to \infty.
	\end{equation} 
	By using these above limits for $e_n^1$ and $e_n^2$, we have \eqref{eq:Q-explicit-one-dim-local} from \eqref{eq:Tn-i-decom}. Finally, for \ref{item:x-great-xn}, one has
	$\Omega_n^- = (x_n, x)$ and $\Omega_n^{+} = \emptyset$
	for $n$ large enough. Similarly to the case \ref{item:x-less-xn}, we also have \eqref{eq:Q-explicit-one-dim-local}.
\end{proof}

Exploiting standard coordinate transform arguments, the result in \cref{lem:curvature-explicit-form-one-dim} can be lifted to the $N$-dimensional case.
\begin{theorem}[Explicit formulation for $Q$]
	\label{thm:Q-explicit-form}
	Let  $\{t_n \} \in c_0^+$ be arbitrary. 
	Assume that \eqref{eq:structural-restricted} and \eqref{eq:boundary-intersection-empty} are fulfilled. 
	{
	Let $y_n \in H^1_0(\Omega) \cap C(\overline\Omega)$ be such that
	\begin{equation}
		\label{eq:yn-H1-convergence-general}
		\norm{y_n - \bar y}_{C(\overline\Omega)} \leq Ct_n  \quad \text{and} \quad \frac{y_n - \bar y}{t_n} \rightharpoonup z \quad \text{weakly in} \quad H^1_0(\Omega)
	\end{equation}
	for some $z \in C(\overline\Omega)$, for some constant $C>0$, and for all $n \geq 1$. 
}%
	Then, 
	{
	for any $p \in C_0(\overline\Omega)$ 
}%
	and for all $1 \leq i \leq K$, there holds
	\begin{equation} \label{eq:Q-component-explicit-general}
		\frac{1}{t_n^2} \int_\Omega p T_n^i \dx  \to -\frac{1}{2} \int_{ \{ \bar y = \tau_i\} } \frac{p(x)z(x)^2}{| \nabla \bar y(x)|} \dH^{N-1}(x)
	\end{equation}
	with $T_n^i$ given in \eqref{eq:Tn-i}.
	
	In particular, 
	{
		for any $h \in L^\barr(\Omega)$, 
}%
	the following formulas hold
	\begin{equation} \label{eq:Q-title-explicit}
		\tilde{Q}(\bar u, \bar p; h)  = \frac{1}{2} \sum_{i=1}^K \sigma_i \int_{ \{ \bar y = \tau_i \} } \frac{\bar p (S'(\bar u)h)^2}{|\nabla \bar y|} \dH^{N-1} 
	\end{equation} 
	and 
	\begin{multline}
		\label{eq:Q-explicit}
		Q(\bar u, \bar p; h) =  \int_\Omega[ \frac{\partial^2 L}{\partial y^2}(x, \bar y) (S'(\bar u) h)^2 + \nu h^2 ] \dx \\
		\begin{aligned}[t]
			- \int_\Omega \1_{\{ \bar y \notin E_f \}} \bar p f''_{yy}( \bar y)(S'(\bar u) h)^2 \dx + \sum_{i=1}^K \sigma_i \int_{ \{ \bar y = \tau_i \} } \frac{\bar p (S'(\bar u)h)^2}{|\nabla \bar y|} \dH^{N-1}. 
		\end{aligned} 
	\end{multline}
	Here $\bar p$ is the unique adjoint state defined in \cref{thm:1st-OS} and \cref{cor:projections}.
\end{theorem}
\begin{proof}
	{
	We first prove \eqref{eq:Q-component-explicit-general}.
}%
	For that purpose, we  consider the following two cases.
	
	\noindent{\emph{Case 1: $\tau_i \neq 0$}.}
	We start with a local result.
	In this case,  $\{ \bar y = \tau_i \}$ is a compact subset of $\Omega$.
	Fixing $x \in \{ \bar y = \tau_i \}$ 
	{
		and assuming w.l.o.g. that $\partial_{N} \bar y(x) >0$, 
}%
	there exist a bounded open set $B \subset \Rbb^{N-1}$, an open interval $I=(a,b)$ and a function $\gamma \in C^1(\overline B)$ satisfying
	\begin{subequations}
		\label{eq:coordinate-local}
		\begin{align}
			& x = (\bar s, \gamma(\bar s)) \in B \times I, \quad \overline{B \times I} \subset \Omega, \quad \{ \bar y = \tau_i \} \cap (B \times I) = \{(s,\gamma(s)) \mid s \in B \}, \label{eq:coordinate-local-1} \\
			& W:= \{(s,t) \mid s \in B, |t-\gamma(s)| < \epsilon \} \subset B \times I, \quad \partial_{N} \bar y\mid_{B\times I} \geq m_0 >0 \label{eq:coordinate-local-2} 
		\end{align}
	\end{subequations}
	for some $\epsilon >0, m_0 >0$, and $\bar s\in B$; 
	{
	see, e.g. \cite[Form.~(3.7)]{WachsmuthWachsmuth2022} and the proof of Lemma 6.10 in \cite{ChristofWachsmuth2018}.
}%
	Defining the sets
	\begin{equation} \nonumber
		\Omega_{i,n}^+ := \{ \bar y \in ( \tau_i, \tau_{i+1} ),  y_n \in (\tau_i - \epsilon_0, \tau_i) \} \, \text{and} \, \Omega_{i,n}^{-} :=  \{ \bar y \in (\tau_{i-1}, \tau_i), y_n \in (\tau_i , \tau_{i} +\epsilon_0) \}
	\end{equation} 
	yields from \eqref{eq:Tn-i} that
	\begin{equation} \nonumber
		T_n^i = (y_n - \tau_i) (\1_{\Omega_{i,n}^+} -\1_{\Omega_{i,n}^-} ).
	\end{equation} 
	Setting $A_n := (\Omega_{i,n}^+ \cup \Omega_{i,n}^-) \cap (B \times I)$ and taking any $s \in \Rbb^{N-1}$, the section $A_n^s$ of $A_n$ determined by $s$ is then defined as follows
	\begin{equation} \nonumber
		A_n^s := \{ t \in \Rbb \mid (s,t) \in A_n \} = 
		\left\{
		\begin{aligned}
			&\{ t \in I \mid (s,t) \in A_n \} && \text{if } s \in B,\\
			& \emptyset && \text{otherwise}.
		\end{aligned}
		\right. 
	\end{equation} 
	The Fubini theorem thus implies that
	\begin{align} \label{eq:T-n-lim-local}
		\frac{1}{t_n^2} \int_{(\Omega_{i,n}^+ \cup \Omega_{i,n}^-) \cap (B \times I)} p T_n^i \dx & = \frac{1}{t_n^2} \int_{A_n} p T_n^i \dx \nonumber  \\
		& =  \int_B d\lambda^{N-1}(s) \frac{1}{t_n^2}  \int_{A_n^s} p(s,t) T_n^i(s,t) d\lambda^1(t).
	\end{align}
	{
	On the other hand, there holds, for any $s \in B$ and $n \geq 1$, that
	\[
		\1_{A_n^s}(t)T_n^i(s,t) = T_n^i(s,t)\1_{I}(t). 
	\]
	}%
	Applying \cref{lem:curvature-explicit-form-one-dim}, one has for all $s \in B$ that
	\begin{equation*}
		\frac{1}{t_n^2}  \int_{A_n^s} p(s,t) T_n^i(s,t) d\lambda^1(t) \to - \frac{1}{2}\frac{p(s,\gamma(s)) z(s,\gamma(s))^2}{|\partial_N \bar y(s,\gamma(s))|} \quad \text{as } n \to\infty. 
	\end{equation*}
	In order to pass to limit in the outer integral of \eqref{eq:T-n-lim-local}, we now prove that
	\begin{equation} \label{eq:Tn-local-bound}
		|\frac{1}{t_n^2}  \int_{A_n^s} p(s,t) T_n^i(s,t) d\lambda^1(t)| \leq C \quad \text{for all $n$ large enough,} 
	\end{equation}
	for all $s \in B$ and for  some positive constant $C$. To that end, 
	{
	setting $\epsilon_n := \norm{y_n - \bar y}_{C(\overline\Omega)}$ yields $\epsilon_n \leq C t_n$, due to the estimate in \eqref{eq:yn-H1-convergence-general}.
}%
	We observe from \eqref{eq:T-n-esti-almost} that
	\begin{equation} \nonumber
		|T_n^i(s,t)| \leq \epsilon_n \1_{ \{t \in I : |\bar y(s,t) - \tau_i | < \epsilon_n \} }
	\end{equation} 
	for all $s \in B$. On the other hand, the mean value theorem gives
	\begin{align*}
		\bar y(s,t) - \tau_i = \bar y(s,t) - \bar y(s, \gamma(s)) = \partial_{N} \bar y(s, \gamma( s) + \theta (t- \gamma( s)))(t -\gamma( s)), \quad \theta \in (0,1).
	\end{align*}
	Combining this with  the inequality in \eqref{eq:coordinate-local-2} yields
	$
	|\bar y(s,t) - \tau_i| \geq m_0 |t -\gamma( s)|,
	$
	which gives 
	\begin{equation} \nonumber
		\{t \in I : |\bar y(s,t) - \tau_i | < \epsilon_n \} \subset [\gamma( s) - \frac{\epsilon_n}{m_0}, \gamma( s) +  \frac{\epsilon_n}{m_0} ].
	\end{equation} 
	We thus derive
	\begin{equation*} 
		|\frac{1}{t_n^2}  \int_{A_n^s} p(s,t) T_n^i(s,t) d\lambda^1(t)| \leq \frac{2}{m_0} \norm{p}_{C(\overline\Omega)} \frac{\epsilon_n^2}{t_n^2}  
	\end{equation*}
	and then \eqref{eq:Tn-local-bound} follows. We now apply the dominated convergence theorem to the outer integral in \eqref{eq:T-n-lim-local} to obtain
	\begin{equation} \label{eq:Tn-i-lim-local}
		\frac{1}{t_n^2} \int_{(\Omega_{i,n}^+ \cup \Omega_{i,n}^-) \cap (B \times I)} p T_n^i \dx \to - \frac{1}{2} \int_B \frac{p(s,\gamma(s)) z(s,\gamma(s))^2}{|\partial_N \bar y(s,\gamma(s))|} d\lambda^{N-1}(s).
	\end{equation}
	Differentiating $\bar y(s,\gamma(s)) =\tau_i$ for all $s \in B$ yields 
	\begin{equation} \nonumber
		\partial_{j} \bar y(s,\gamma(s)) = - \partial_{N}\bar y(s,\gamma(s))	\partial_{j}\gamma(s) \quad \text{for all } 1 \leq j \leq N-1,
	\end{equation} 
	which infers $|\nabla \bar y(s, \gamma(s))| = |\partial_N \bar y(s,\gamma(s))| {(1+\sum_{j=1}^{N-1} |	\partial_{j}\gamma(s)|^2)^{1/2}}$ for all $s \in B$.
	We therefore derive
	\begin{multline*}
		\int_B \frac{p(s,\gamma(s)) z(s,\gamma(s))^2}{|\partial_N \bar y(s,\gamma(s))|} d\lambda^{N-1}(s) \\
		\begin{aligned}
			& = \int_B \frac{p(s,\gamma(s)) z(s,\gamma(s))^2}{|\nabla \bar y(s, \gamma(s))|}{(1+\sum_{j=1}^{N-1} |	\partial_{j}\gamma(s)|^2)^{1/2}} d\lambda^{N-1}(s)  = \int_{\{ \bar y = \tau_i \} \cap (B\times I)} \frac{p z^2}{|\nabla \bar y|} \dH^{N-1}.
		\end{aligned}
	\end{multline*}
	Combining this with \eqref{eq:Tn-i-lim-local} gives
	\begin{equation*} 
		\frac{1}{t_n^2} \int_{(\Omega_{i,n}^+ \cup \Omega_{i,n}^-) \cap (B \times I)} p T_n^i \dx \to - \frac{1}{2} \int_{\{ \bar y = \tau_i \} \cap (B\times I)} \frac{p z^2}{|\nabla \bar y|} \dH^{N-1}.
	\end{equation*}
	Since the set $\{ \bar y = \tau_i \}$ is a compact subset of $\Omega$, then it is covered by a finitely many sets $B\times I$ of type \eqref{eq:coordinate-local}. Then by using a standard partition-of-unity argument, we have \eqref{eq:Q-component-explicit-general} when $\tau_i \neq 0$.

	\noindent{\emph{Case 2: $\tau_i = 0$}.}
	In this situation, we have
	\begin{equation}
		\label{eq:boundary-intersection-2}
		\{ \bar y = 0 \} = \partial\Omega \cup ( \{ \bar y = 0 \} \backslash \partial\Omega ) \quad \text{and} \quad \textrm{dist}(\partial \Omega, (\{ \bar y =0 \} \backslash \partial \Omega)) > 0,
	\end{equation}
	where the last condition is identical to that in \eqref{eq:boundary-intersection-empty}.
	For any $r >0$ and a subset $V \subset \overline\Omega$, we define the  open set $	V^r := \{ s \in \Omega \mid \text{dist}(s,V) < r  \}$,
	where dist$(s,V)$ is the distance from $s$ to $V$. 
	Take $\epsilon >0$ be arbitrary.
	Thanks to the second condition in \eqref{eq:boundary-intersection-2} and due to 
	{
	the closedness of subsets 
	}%
	$\partial\Omega$ and $\{ \bar y = 0 \} \backslash \partial\Omega$ as well as the fact that 
	{
	$p \in C_0(\overline\Omega)$, 
	there exists an $\epsilon_* >0$ such that
	\begin{equation*}
	(\partial\Omega)^{\epsilon_*} \cap ( \{ \bar y = 0 \} \backslash \partial\Omega )^{\epsilon_*} = \emptyset \quad \text{and} \quad \max\{|p(x)|: x \in \overline{(\partial\Omega)^{\epsilon_*}} \}< (C^2c_s\sigma_{\max})^{-1}\epsilon
	\end{equation*}
	with $c_s$ being the constant in \eqref{eq:structure-non-diff}, $\sigma_{\max}$ determined as in \cref{lem:Q-defined}, and $C$ given in \eqref{eq:yn-H1-convergence-general}.
}%
	Since the set $\{ \bar y = 0 \} \backslash \partial\Omega$ is compact in $( \{ \bar y = 0 \} \backslash \partial\Omega )^{\epsilon_*}$, we deduce from the argument analogous to Case 1 that
	\begin{equation} \label{eq:Tn-zero-lim-nonboundary}
	\frac{1}{t_n^2} \int_{( \{ \bar y = 0 \} \backslash \partial\Omega )^{\epsilon_*}} p T_n^i \dx \to - \frac{1}{2} \int_{\{ \bar y = 0 \} \backslash \partial\Omega} \frac{p z^2}{|\nabla \bar y|} \dH^{N-1}.
	\end{equation}	
	On the other hand, similar to \eqref{eq:zeta-esti}, we have an estimate for the integral over $(\partial\Omega)^{\epsilon_*}$ as follows
	\begin{equation} \nonumber
	|\int_{(\partial\Omega)^{\epsilon_*}} p T_n^i| \leq \sigma_{\max} c_s \norm{p}_{C(\overline{(\partial\Omega)^{\epsilon_*}})} \norm{y_n - \bar y}_{C(\overline{(\partial\Omega)^{ \epsilon_*}} )}^2,
	\end{equation} 
	which, along with the estimate in \eqref{eq:yn-H1-convergence-general}, infers
	\begin{equation} \nonumber
		\limsup_{n \to \infty} \frac{1}{t_n^2} |\int_{(\partial\Omega)^{\epsilon_*}} p T_n^i| \leq \sigma_{\max} c_sC^2  \norm{p}_{C(\overline{(\partial\Omega)^{\epsilon_*}})} < \epsilon.
	\end{equation} 
	Since $\epsilon >0$ is arbitrary, there holds
	\begin{equation} \nonumber
	\limsup_{n \to \infty} \frac{1}{t_n^2} |\int_{(\partial\Omega)^{\epsilon_*}} p T_n^i| = 0 = - \frac{1}{2} \int_{\partial\Omega} \frac{p z^2}{|\nabla \bar y|} \dH^{N-1}.
	\end{equation}  
	Combining this with \eqref{eq:Tn-zero-lim-nonboundary} yields \eqref{eq:Q-component-explicit-general} for the situation where $\tau_i =0$.
	
	We have shown \eqref{eq:Q-component-explicit-general}.
	
	\medskip 
	
	{
	It remains to show \eqref{eq:Q-title-explicit} and \eqref{eq:Q-explicit}. To this end, taking $h \in L^\barr(\Omega)$ and $\{t_n \} \in c_0^+$  arbitrarily, we set  $y_n:= S(\bar u + t_n h)$ and $z := S'(\bar u)h$. Thanks to \cref{prop:control-to-state,prop:G-diff-control2state}, $y_n$ and $z$ satisfy \eqref{eq:yn-H1-convergence-general}. 
	From this and the definitions of $\tilde{Q}$ and $Q$, we have  \eqref{eq:Q-title-explicit} and \eqref{eq:Q-explicit}, due to \eqref{eq:Q-component-explicit-general} and \eqref{eq:zeta-func-expression}.
}%
\end{proof}

We now combine \cref{thm:2nd-OS-nec,thm:2nd-OS-suf} with \cref{thm:Q-explicit-form} to arrive at the following explicit second-order necessary and sufficient conditions for \eqref{eq:P}, which are the main results of this paper.
\begin{theorem}[explicit second-order necessary optimality conditions]
	\label{thm:2nd-OS-nec-explicit}
	Let $\bar u$ be a local minimizer of \eqref{eq:P} such that the associated state $\bar y:= S(\bar u)$ fulfills  \eqref{eq:structural-restricted} and \eqref{eq:boundary-intersection-empty}.  
	{
	Then there exist a unique adjoint state $\bar p \in H^1_0(\Omega) \cap C(\overline\Omega)$ and a unique multiplier $\bar \lambda \in \partial j(\bar u)$ satisfying
}%
	\eqref{eq:1st-OS}  for $\chi(\cdot) = \1_{\{ \bar y \notin E_f \}}(\cdot) f'(\bar y(\cdot))$. Moreover, the following second-order necessary optimality condition holds:
	\begin{multline*}
		Q(\bar u, \bar p; h) =  \int_\Omega[ \frac{\partial^2 L}{\partial y^2}(x, \bar y) (S'(\bar u) h)^2 + \nu h^2 ] \dx - \int_\Omega \1_{\{ \bar y \notin E_f \}} \bar p f''_{yy}( \bar y)(S'(\bar u) h)^2 \dx  \\
		\begin{aligned}[t]
		+ \sum_{i=1}^K \sigma_i \int_{ \{ \bar y = \tau_i \} } \frac{\bar p (S'(\bar u)h)^2}{|\nabla \bar y|} \dH^{N-1} \geq 0  \qquad\text{for all }h\in \mathcal{C}_{L^\barr(\Omega)}({U}_{ad};\bar u).
		\end{aligned} 
	\end{multline*}
\end{theorem}
\begin{theorem}[explicit second-order sufficient optimality conditions] 
	\label{thm:2nd-OS-suf-explicit}
	{
		Assume that $N \in \{{1,2,3}\}$.	
}%
	Let $\bar u$ be an admissible point of \eqref{eq:P} for which conditions \eqref{eq:structural-restricted} and \eqref{eq:boundary-intersection-empty} are satisfied by $\bar y:= S(\bar u)$. 
	Assume further that there exist an adjoint state $\bar p \in H^1_0(\Omega) \cap C(\overline\Omega)$ and a multiplier $\bar \lambda \in \partial j(\bar u)$ satisfying \eqref{eq:1st-OS}  for $\chi(\cdot) = \1_{\{ \bar y \notin E_f \}}(\cdot) f'(\bar y(\cdot))$. If the following explicit second-order sufficient condition is verified:
	\begin{multline*}
		Q(\bar u, \bar p; h) =  \int_\Omega[ \frac{\partial^2 L}{\partial y^2}(x, \bar y) (S'(\bar u) h)^2 + \nu h^2 ] \dx - \int_\Omega \1_{\{ \bar y \notin E_f \}} \bar p f''_{yy}( \bar y)(S'(\bar u) h)^2 \dx  \\
		\begin{aligned}[t]
		+ \sum_{i=1}^K \sigma_i \int_{ \{ \bar y = \tau_i \} } \frac{\bar p (S'(\bar u)h)^2}{|\nabla \bar y|} \dH^{N-1} > 0  \qquad\text{for all }h\in \mathcal{C}_{L^2(\Omega)}({U}_{ad};\bar u)\setminus \{0\},
	\end{aligned} 
	\end{multline*}
	then the quadratic growth condition \eqref{eq:quadratic-grownth} is valid for some positive constants $c, \rho >0$.
\end{theorem}

\begin{remark}
	\label{rem:bang-bang-problem}
	For bang-off-bang optimal control problems, i.e., $\nu =0$, we expect to derive second-order necessary and sufficient conditions based on the curvature functional $Q$ and the critical cone $\mathcal{C}_{L^\barr(\Omega)}(U_{ad};\bar u)$ introduced in \eqref{eq:critical-cone}. However, in this situation, the critical cone  $\mathcal{C}_{L^\barr(\Omega)}(U_{ad};\bar u)$  
	{
	is frequently reduced to 
}%
	$\{0\}$; see, e.g. \cite[Rem.~3.2 \& Prop.~3.3]{Casas2012} and thus the second-order optimality conditions as in \cref{thm:2nd-OS-nec,thm:2nd-OS-suf} do not provide any useful information. 	
	On the other hand, following the approach used in \cite{ChristofWachsmuth2018} and \cite{WachsmuthWachsmuth2022}, where the problems with smooth PDEs were investigated, we might extend the functional $Q$ on the space of signed finite Radon measures $\mathcal{M}(\Omega)$, the dual space of the space of all continuous functions vanishing on the boundary $\partial \Omega$. 
	It is noted that, in order to derive the second-order optimality conditions for problems considered in \cite{ChristofWachsmuth2018} and \cite{WachsmuthWachsmuth2022}, an additional structural assumption imposed on the adjoint state $\bar p$ is required and guarantees a nondegeneracy condition; see, \cite[Ass.~6.6]{ChristofWachsmuth2018} and \cite[Ass.~ 4.12]{WachsmuthWachsmuth2022}. However, in our approach, we do not impose any structural assumption on the adjoint state $\bar p$. 
\end{remark}

\section{Stability analysis with respect to $\kappa \to 0$} \label{sec:stability}
In this section, we use the notation $u^*_\kappa$ and $u^*$ for any minimizer of \eqref{eq:P} when $\kappa >0$ and $\kappa =0$, respectively.
We also set $y^*_\kappa := S(u^*_\kappa)$ for $\kappa >0$ and $y^* := S(u^*)$. 
We shall study the behavior of a sequence of local minimizers  $u^*_\kappa$ of \eqref{eq:P} as $\kappa \to 0^+$. 

\begin{theorem}
	\label{thm:stability-convergence}
	Let $\{ u^*_\kappa \}_{\kappa >0}$ be a sequence of local minimizers to \eqref{eq:P}. 
	Then, there exists a subsequence 
	{
	$\{u^*_{\kappa_n}\}_{n \geq 1}$ weakly-$*$ converging in $L^\infty(\Omega)$ to an element $u^* \in U_{ad}$.
}%
	Moreover, there hold
	\begin{enumerate}[label=(\roman*)]
		\item \label{item:weak-convergence} $u^*$ is a solution to $(P_0)$;
		\item \label{item:Linfty-convergence} $u^*_{\kappa_n} \to u^*$ in $L^\infty(\Omega)$.
	\end{enumerate}
\end{theorem}
\begin{proof}
	The existence of a subsequence $\{u^*_{\kappa_n}\}_{n \geq 1}$  and its weak-$*$ convergence to $u^* \in U_{ad}$ follows from the boundedness in $L^\infty(\Omega)$ of the admissible set $U_{ad}$. Let us show \ref{item:weak-convergence}. To this end, taking any $u \in U_{ad}$ yields
	\begin{equation}
		\label{eq:minimum-esti}
		\begin{aligned}
			J_0(u^*) &\leq \liminf\limits_{n \to \infty} J_0(u^*_{\kappa_n}) \leq \liminf\limits_{n \to \infty} [ J_0(u^*_{\kappa_n}) + \kappa_n \norm{u^*_{\kappa_n}}_{L^1(\Omega)} ] \\
			&= \liminf\limits_{n \to \infty} J_{\kappa_n}(u^*_{\kappa_n}) \leq  \liminf\limits_{n \to \infty} J_{\kappa_n}(u) = J_0(u).
		\end{aligned}
	\end{equation}	
	This shows \ref{item:weak-convergence}. To prove \ref{item:Linfty-convergence}, we now use \cref{thm:1st-OS} and \cref{cor:projections}. 
	According to \cref{thm:1st-OS}, there exist $\bar p := p^*_{\kappa_n}$, $\chi := \chi^*_{\kappa_n} \in L^\infty(\Omega)$, and $\bar \lambda := \lambda^*_{\kappa_n} \in L^\infty(\Omega)$ with $\lambda^*_{\kappa_n} \in \partial j(u^*_{\kappa_n})$ satisfying \eqref{eq:1st-OS} for $\bar u:= u_{\kappa_n}^*$ and $\bar y := y^*_{\kappa_n}$.
	In particular, we have \eqref{eq:adjoint-OS}, \eqref{eq:normal-OS} and \eqref{eq:projection-u}, that is, there hold
	\begin{align}
		& A^* p^*_{\kappa_n} +  \chi^*_{\kappa_n}  p^*_{\kappa_n} = \frac{\partial L}{\partial y}(x, y^*_{\kappa_n}) \, \text{in } \Omega, \quad 	p^*_{\kappa_n} =0 \, \text{on } \partial\Omega, \label{eq:adjoint-OS-stability}  \\
		& \int_\Omega p^*_{\kappa_n} u^*_{\kappa_n} + \nu |u^*_{\kappa_n}|^2 + \kappa_n \lambda^*_{\kappa_n} u^*_{\kappa_n} dx \leq \int_\Omega (p^*_{\kappa_n}  + \nu u^*_{\kappa_n} + \kappa_n \lambda^*_{\kappa_n})u dx \, \forall  u \in U_{ad}, \label{eq:variational-ine}
	\end{align}
	and
	\begin{equation}
		\label{eq:projection-u-stability}
		u^*_{\kappa_n}(x) = \proj_{[\alpha,\beta]}( - \frac{1}{\nu} (p^*_{\kappa_n}(x) + \kappa_n \lambda^*_{\kappa_n}(x)) ) \quad \text{for a.a. } x \in \Omega.
	\end{equation}
	Since $u^*_{\kappa_n}$ weakly-$*$ converges to $u^*$, the sequence $\{y^*_{\kappa_n} \}$ strongly converges to $y^*$ in $H^1_0(\Omega) \cap C(\overline\Omega)$.
	Moreover, thanks to the fact that $|\lambda^*_{\kappa_n}| \leq 1$ a.a. in $\Omega$, there then exists a subsequence, denoted in the same way, such that
	$\lambda^*_{\kappa_n} \rightharpoonup^* \lambda^*$ for some $\lambda^* \in L^\infty(\Omega)$.  From the weak-to-strong convergence of solutions to \eqref{eq:adjoint-OS-stability}, there holds
	\begin{equation}
		\label{eq:adjoint-state-limit}
		p^*_{\kappa_n} \to p^* \quad \text{in } H^1_0(\Omega) \cap C(\overline\Omega)
	\end{equation}
	for some $p^* \in H^1_0(\Omega) \cap C(\overline\Omega)$. 	
	Letting now $\kappa_n \to 0$ in \eqref{eq:variational-ine} and using the weak lower semicontinuity of the $L^2$-norm yield
	\begin{equation} \label{eq:variational-ine-u-star}
		\int_\Omega (p^* + \nu u^*)(u - u^*) dx \geq 0 \quad \text{for all } u \in U_{ad},
	\end{equation} 
	or, equivalently,
	$
		u^*(x) = \proj_{[\alpha,\beta]}\left( - \frac{1}{\nu} p^*(x) \right) \, \text{for a.a. } x \in \Omega.
	$ 
	From this and \eqref{eq:projection-u-stability}, we deduce from the Lipschitz continuity with modulus $1$ of the projection mapping that
	\begin{equation} \nonumber
		|u^*_{\kappa_n}(x) - u^*(x)| \leq \frac{1}{\nu}[|p^*_{\kappa_n}(x) - p^*(x)| + \kappa_n |\lambda^*_{\kappa_n}(x) | \leq \frac{1}{\nu}[|p^*_{\kappa_n}(x) - p^*(x)| + \kappa_n   ]
	\end{equation} 
	for a.a. $x \in \Omega$. Combining this with \eqref{eq:adjoint-state-limit} gives \ref{item:Linfty-convergence}.
\end{proof}

Let us prove the converse result on the  $L^2(\Omega)$- convergence.
\begin{theorem}
	\label{thm:stability-convergence-inverse}
	Assume that $u^*$ is a 
	{
	strict local minimizer of 
}%
	$(P_0)$ in the $L^2(\Omega)$-sense, that is, there exists an $\rho >0$ such that
	\begin{equation}
		\label{eq:strict-minimizer}
		J_0(u^*) < J_0(u) \quad \text{for all } u \in U_{ad} \cap \overline B_{L^2(\Omega)}(u^*, \rho) \backslash \{ u^*\}.
	\end{equation}
	Then there exists a sequence $\{ u^*_{\kappa}\}$ of local minima of \eqref{eq:P} that converges strongly to $u^*$ in $L^2(\Omega)$.
\end{theorem}
\begin{proof}
	For any $\kappa >0$, we consider the problems
	\begin{equation}
		\label{eq:auxiliary-problems}
		\tag{P$_{\kappa}^\rho$}
		\min\{ J_\kappa(u) \mid u \in U_{ad} \cap \overline B_{L^2(\Omega)}(u^*, \rho)  \}.
	\end{equation}
	Similar to \cref{prop:existence}, for every $\kappa >0$, \eqref{eq:auxiliary-problems} admits at least one solution $u^*_\kappa$. 
	The boundedness of $U_{ad}$ in $L^\infty(\Omega)$  implies that there exists a subsequence of $\{u^*_\kappa\}_{\kappa >0}$ which weakly-$*$ converges to some $\tilde{u} \in U_{ad} \cap \overline B_{L^2(\Omega)}(u^*, \rho)$. Similar to \eqref{eq:minimum-esti}, we have $J_0(\tilde{u}) \leq J_0(u)$ for all $u \in U_{ad} \cap \overline B_{L^2(\Omega)}(u^*, \rho)$ and thus $\tilde{u}$ is a local minimizer of $(P_0)$. Since $u^*$ is a unique solution to \eqref{eq:strict-minimizer}, one has $\tilde{u} = u^*$ and thus the full sequence $\{u^*_\kappa\}_{\kappa>0}$ weakly-$*$ converges to $u^*$.   
	On the other hand, analogous to \cref{thm:1st-OS}, 
	there exist $\bar p := p^*_{\kappa}$, $\chi := \chi^*_{\kappa} \in L^\infty(\Omega)$, and $\bar \lambda := \lambda^*_{\kappa} \in L^\infty(\Omega)$ with $\lambda^*_{\kappa} \in \partial j(u^*_{\kappa})$ satisfying \eqref{eq:adjoint-OS}, \eqref{eq:Clarke-multiplier} for $\bar u:= u^*_\kappa, \bar y:= y^*_\kappa$ and
	\begin{equation} \label{eq:variational-ine-stability}
		\int_\Omega (p^*_{\kappa} + \nu u^*_\kappa + \kappa \lambda^*_\kappa)(u - u^*_\kappa ) \geq 0 \quad \text{for all } u \in U_{ad} \cap \overline B_{L^2(\Omega)}(u^*, \rho),
	\end{equation}
	or, equivalently,
	\[
		\int_\Omega (\frac{1}{\nu}(-p^*_{\kappa}  - \kappa \lambda^*_\kappa) -  u^*_\kappa)(u - u^*_\kappa ) \leq 0 \quad \text{for all } u \in U_{ad} \cap \overline B_{L^2(\Omega)}(u^*, \rho).
	\]
	The Hilbert projection theorem then implies that 
	\begin{equation}
		\label{eq:projection-stability}
		u^*_\kappa = \proj_{U_{ad} \cap \overline B_{L^2(\Omega)}(u^*, \rho)}\frac{1}{\nu}(-p^*_{\kappa}  - \kappa \lambda^*_\kappa).
	\end{equation}
	Moreover, similar to the arguments in the proof of \cref{thm:stability-convergence}, there exists a subsequence $\{p^*_{\kappa_n}\}$ that satisfies the limit \eqref{eq:adjoint-state-limit} for some $p^* \in H^1_0(\Omega) \cap C(\overline\Omega)$. Plugging $\kappa := \kappa_n$ into \eqref{eq:variational-ine-stability} and thus letting $n \to \infty$ yield
	\[
		\int_\Omega (p^* + \nu u^*)(u-u^*) \geq 0 \quad \text{for all } u \in U_{ad} \cap \overline B_{L^2(\Omega)}(u^*, \rho)
	\]
	(see arguments showing \eqref{eq:variational-ine-u-star}). From the Hilbert projection theorem, we then deduce
	\begin{equation*}
		u^* = \proj_{U_{ad} \cap \overline B_{L^2(\Omega)}(u^*, \rho)}\frac{1}{\nu}(-p^*).
	\end{equation*}
	Combining this with \eqref{eq:projection-stability} for $\kappa = \kappa_n$ yields
	\begin{align*}
		\norm{u^*_{\kappa_n} - u^*}_{L^2(\Omega)} \leq \frac{1}{\nu}[\norm{p^*_{\kappa_n} - p^*}_{L^2(\Omega)} + \kappa_n \norm{\lambda^*_{\kappa_n}}_{L^2(\Omega)}] \to 0, 
	\end{align*}
	due to the limit \eqref{eq:adjoint-state-limit} and the fact that $|\lambda_{\kappa_n}^*| \leq 1$ a.a. in $\Omega$ (see \eqref{eq:subderivative-j-formulation}). 
	We have just shown that the subsequence $\{u^*_{\kappa_n}\}$ strongly converges to $u^*$ in $L^2(\Omega)$. We then have $\norm{u^*_\kappa - u^*}_{L^2(\Omega)} \to 0$ by using the subsequence-subsequence argument. Consequently, $\norm{u^*_\kappa - u^*}_{L^2(\Omega)} < \rho$ for $\kappa$ small enough and thus $u^*_\kappa$ is a local minimizer of \eqref{eq:P}.
\end{proof}

\begin{theorem} \label{thm:stability-Lipschitz}
	Let $u^*$ be a local minimizer to $(P_0)$ and let $\{u^*_\kappa \}_{\kappa >0}$ be the sequence of local minimizers to \eqref{eq:P} as defined in \cref{thm:stability-convergence-inverse}. Assume that the following quadratic growth condition holds
	\begin{equation}
		\label{eq:quadratic-grownth-stability}
		J_0(u^*) + c \norm{u - u^*}_{L^2(\Omega)}^2 \leq J_0(u) \quad \text{for all } u \in {U}_{ad} \cap \overline B_{L^2(\Omega)}(u^*, \rho)
	\end{equation}
	with some constants $c>0, \rho >0$. Then, 
	{
	there exist constants 
}%
	$\kappa_0 >0$ and $C>0$ such that
	\begin{equation}
		\label{eq:Lipschitz-L2}
		\norm{u^*_\kappa - u^*}_{L^2(\Omega)} \leq C \kappa \quad \text{for all } 0 < \kappa < \kappa_0.
	\end{equation}
\end{theorem}
\begin{proof}
	Thanks to \cref{thm:stability-convergence-inverse}, there exists a constant $\kappa_0 >0$ satisfying  $\norm{u^*_\kappa - u^*}_{L^2(\Omega)} < \rho$ for all $0 < \kappa < \kappa_0$.  The quadratic growth condition \eqref{eq:quadratic-grownth-stability}  then implies for all $\kappa \in (0, \kappa_0)$ that
	\begin{align*}
		J_0(u^*) + c \norm{u^*_\kappa - u^*}_{L^2(\Omega)}^2 + \kappa \norm{u^*_\kappa}_{L^1(\Omega)} & \leq J_0(u^*_\kappa) + \kappa \norm{u^*_\kappa}_{L^1(\Omega)} \\
		&  = J_\kappa(u^*_\kappa) \leq J_\kappa(u^*) = J_0(u^*) + \kappa \norm{u^*}_{L^1(\Omega)}.
	\end{align*}
	We then deduce for all $\kappa \in (0, \kappa_0)$ that
	\begin{align*}
		c \norm{u^*_\kappa - u^*}_{L^2(\Omega)}^2 & \leq \kappa (\norm{u^*}_{L^1(\Omega)} - \norm{u^*_\kappa}_{L^1(\Omega)} ) \\
		&   \leq \kappa \norm{u^* - u^*_\kappa}_{L^1(\Omega)} \leq \kappa (\lambda^N(\Omega))^{1/2} \norm{u^* - u^*_\kappa}_{L^2(\Omega)},
	\end{align*}
	where we have just employed the Cauchy--Schwarz inequality to get the last estimate. The desired inequality \eqref{eq:Lipschitz-L2} then follows. 
\end{proof}

\section{Conclusions}
We have derived second-order optimality conditions as well as the Lipschitz stability of solutions for non-smooth semilinear elliptic optimal control problems involving the $L^1$-norm of the control in the cost functional. The control-to-state operator is in general directionally differentiable only. By using the regularization approach, a system of first-order optimality conditions involving the Clarke subdifferential of the non-smooth nonlinear coefficient in the state equation has been obtained. 
Under the structural assumption imposed on the considered state, the control-to-state operator is then G\^{a}teaux-differentiable and the so-called curvature functional $Q$ for a component of the objective has been introduced. Based on this curvature and a Taylor-type expansion, the second-order necessary and sufficient optimality conditions have been derived by using some key lemmas. 
The difference between these conditions is minimal.
Moreover, under a more restrictive structural assumption on the mentioned state, an explicit formulation for the curvature functional $Q$ has been established. Consequently, explicit second-order necessary and sufficient optimality conditions have been presented. These explicit conditions could   be of great importance to  discretization error estimates for the numerical approximation of \eqref{eq:P}; see, e.g. \cite{ClasonNhuRosch2022,ClasonNhuRosch2023_part2}. Such estimates will be studied in a future work. Finally, the convergence of optimal solutions for vanishing the sparsity parameter has been shown and the associated convergence rate has also been proven to be $O(\kappa)$.

\bibliography{nonsmoothsemilinearsparse}


\section*{Statements and Declarations}

\subsection*{Funding}
	{%
		This research is funded by Phenikaa University under grant number PU2022-1-A-14.}
	
\subsection*{Competing Interests}
	{%
		Financial interests: The authors have received travel  support from Phenikaa University.
	}
\subsection*{Author Contributions}
{%
	All authors contributed to the study conception and design. The first draft of the manuscript was written by the first author and all authors commented on previous versions of the manuscript. All authors read and approved the final manuscript.
}

\subsection*{Data Availability}
{%
	All data generated or analysed during this study are included in this published article [and its supplementary information files].
}
\end{document}